\documentclass[a4paper]{article}

\usepackage{amssymb, amsmath, amscd, amsfonts, amsthm}

\usepackage{subcaption}
\usepackage{enumerate}
\usepackage{color}
\usepackage{microtype}

\usepackage{colonequals}
\newcommand{\defeq}{\colonequals}

\usepackage[outline]{contour}
\contourlength{0.1em}

\usepackage{algorithm}
\usepackage{algpseudocode}

\algnewcommand{\IIf}[1]{\State\algorithmicif\ #1\ \algorithmicthen}
\algnewcommand{\EElse}{\algorithmicelse\ }
\algnewcommand{\EndIIf}{\unskip\ }
\algnewcommand{\algorithmicforeach}{\textbf{for each}}
\algdef{S}[FOR]{ForEach}[1]{\algorithmicforeach\ #1\ \algorithmicdo}

\usepackage{tikz}
\usetikzlibrary{calc}
\usetikzlibrary{patterns}
\usetikzlibrary{decorations.pathreplacing}
\usetikzlibrary{decorations.pathmorphing}
\usetikzlibrary{decorations.markings}
\usetikzlibrary{shapes.geometric}

\tikzset{dot/.style={draw,shape=circle,fill=black,scale=0.4}}

\usepackage[hidelinks, pagebackref, pdftex]{hyperref}

\renewcommand*{\backref}[1]{}
\renewcommand*{\backrefalt}[4]{
  \ifcase #1 %
   [No citations.]%
  \else
   [#2]%
  \fi
}

\theoremstyle{plain}
\numberwithin{equation}{section}
\newtheorem{theorem}[equation]{Theorem}
\newtheorem*{theorem*}{Theorem}
\newtheorem{corollary}[equation]{Corollary}
\newtheorem*{corollary*}{Corollary}
\newtheorem{lemma}[equation]{Lemma}
\newtheorem*{lemma*}{Lemma}
\newtheorem{proposition}[equation]{Proposition}
\newtheorem*{proposition*}{Proposition}

\theoremstyle{definition}
\newtheorem{definition}[equation]{Definition}
\newtheorem{remark}[equation]{Remark}
\newtheorem*{definition*}{Definition}
\newtheorem*{keywords}{keywords}
\newtheorem*{acknowledgements}{Acknowledgements}

\newtheoremstyle{dotless}{}{}{}{}{\bfseries}{}{ }{}
\theoremstyle{dotless}
\newtheorem*{ccode}{\textbf{Mathematics Subject Classification (2010):}}

\newcommand{\calC}{\mathcal{C}}
\newcommand{\calF}{\mathcal{F}}

\newcommand{\calT}{\mathcal{T}}
\newcommand{\HH}{\mathbb{H}}
\newcommand{\RR}{\mathbb{R}}
\newcommand{\carriedby}{\prec}

\newcommand{\PTIME}{\textnormal{\textbf{P}}}

\newcommand{\dC}{d_{\calC(S)}}  
\DeclareMathOperator{\Mod}{Mod}  
\DeclareMathOperator{\intersection}{\iota}
\DeclareMathOperator{\canonical}{\sigma}
\DeclareMathOperator{\poly}{poly}

\DeclareMathOperator{\diam}{diam}
\DeclareMathOperator{\short}{short}
\DeclareMathOperator{\translationlength}{\ell}

\title{Polynomial-time algorithms for the curve graph}
\author{Mark C. Bell\footnote{Department of Mathematics, University of Illinois: \texttt{mcbell@illinois.edu}} \and Richard C. H. Webb\footnote{DPMMS, Centre for Mathematical Sciences, University of Cambridge: \texttt{rchw2@cam.ac.uk}}}

\begin{document}

\maketitle

\begin{abstract}
We describe a polynomial-time algorithm to compute a (tight) geodesic between two curves in the curve graph.
As well as enabling us to compute the distance between a pair of curves, this has several applications to mapping classes.
For example, we can use these geodesics to compute:
\begin{itemize}
\item the asymptotic translation length,
\item the Nielsen--Thurston type, and
\item the canonical curve system
\end{itemize}
of a mapping class in polynomial time in its word length.
\end{abstract}

\keywords{curve graph, geodesics, polynomial-time, mapping class group, asymptotic translation length, Nielsen--Thurston classification, canonical curve system.}

\ccode{57M20, 05C12, 37E30}

\section{Introduction}

In the mid 1970s, Thurston showed that any mapping class is periodic, pseudo-Anosov or reducible \cite{Bers, CassonBleiler, FarbMargalit, FLP, HandelThurston, Thurston}.
This theorem is a milestone in the study of the mapping class group $\Mod(S)$ and is known as the Nielsen--Thurston classification.
Thurston also gave a beautiful rephrasing of this from the point of view of $3$--manifolds, namely that the corresponding mapping torus has geometry modelled on $\HH^2 \times \RR$, modelled on $\HH^3$ or has a non-trivial JSJ decomposition \cite{ThurstonHype}.
This is an important step in the formulation of Thurston's geometrization conjecture (now solved by Perelman \cite{PerelmanGeometry, Perelman3Manifolds}) and hyperbolization theorem \cite{ThurstonHype}.

In this paper we give a \emph{polynomial-time} algorithm to decide whether a mapping class is pseudo-Anosov (Algorithm~\ref{alg:nt_type}).
That is, the running time of our algorithm is bounded by a polynomial function of the complexity of its input.
We describe how we represent mapping classes and our notion of complexity in detail in Section~\ref{sub:complexity}.
However a key property of our complexity is that if we fix a generating set of $\Mod(S)$ then the complexity of our representation of a word is coarsely equivalent to the word's length.
Hence:

\begin{theorem}
\label{thrm:nt_type}
Fix an orientable surface $S$.
The pseudo-Anosov decision problem for $\textnormal{Mod}(S)$ lies in \PTIME. \qed
\end{theorem}

This theorem has also been proven independently by Margalit--Strenner--Yurtta\c{s} using a different method.
They consider the piecewise-linear action of a mapping class $f$ on $\mathcal{ML}(S)$, the space of measured laminations.
They show that iteration can be used to find a lamination that is projectively invariant under $f$, and so decide whether $f$ is pseudo-Anosov, in polynomial time.
As a consequence they also obtain the dilatation of $f$ or, equivalently, the translation length of $f$ on the Teichm\"{u}ller space of $S$ equipped with the Teichm\"{u}ller metric.

By contrast, our algorithm relies on considering the natural action of a mapping class $f$ on the \emph{curve graph} $\calC(S)$.
This graph has a vertex for each (essential) curve on $S$ and curves $c, c' \in \calC(S)$ are connected via an edge (of length one) if and only if $\intersection(c, c') = 0$.
This is the $1$--skeleton of the curve complex, which was introduced by Harvey \cite{Harvey}.
We proceed by finding a curve that is close to an axis of $f$ or close to an $f$--invariant multicurve.
This enables us to compute the asymptotic translation length of $f$ on $\calC(S)$ by applying a large power (Algorithm~\ref{alg:asymptotic_translation_length}).
As a consequence, we can decide whether $f$ is pseudo-Anosov.

Our approach is also significantly different to other algorithms for determining Nielsen--Thurston type such as:
\begin{itemize}
\item The Bestvina--Handel algorithm \cite{BestvinaHandel}.
This works by determining the irreducibility of the corresponding outer automorphism induced on $\pi_1(S)$.
However this is only known to run in exponential time in the word length \cite[Corollary~A.2]{Kapovich}.
\item The list-and-check algorithms \cite[Corollary~3.4]{BellReducibility} \cite[Section~4]{HamidiChen} \cite[Theorem~1.3]{KoberdaMangahas}.
These work by producing a list of candidate reducing curves.
However these lists are typically exponentially long in the word length.
The exception to this is for punctured spheres where $\Mod(S_{0,n})$ is closely related to the braid group on $n-1$ strands.
Using the additional Garside structure of the braid group, Calvez refined the algorithm of Benardete--Gutierrez--Nitecki \cite{BenardeteGutierrezNitecki} to a polynomial-time algorithm for reducibility of braids \cite[Theorem~1]{Calvez}.
\item The \texttt{flipper} algorithm \cite{flipper} \cite[Section~4.3]{BellThesis}.
This works by performing maximal train track splits to find an Agol cycle \cite[Corollary~3.4]{Agol}.
However this is not known to run in polynomial time.
\end{itemize}

As a by-product, our techniques also recover the canonical curve system of a mapping class (Algorithm~\ref{alg:canonical}).
Margalit--Strenner--Yurtta\c{s} have also independently constructed such a polynomial-time algorithm using a different method.

\begin{theorem}
\label{thrm:canonical}
Fix an orientable surface $S$.
Algorithm~\ref{alg:canonical} takes a mapping class $f$ as input and returns its canonical curve system $\sigma(f)$ in polynomial time. \qed
\end{theorem}

To prove Theorems~\ref{thrm:nt_type} and \ref{thrm:canonical} we use results of Masur and Minsky \cite{MasurMinskyII, MasurMinskyI, MasurMinskyQuasiconvexity} namely the hyperbolicity of the curve graph, the bounded geodesic image theorem, the notion of tight geodesics and the quasiconvexity of train track splitting sequences.
Additionally, we use a theorem of Bowditch \cite{BowditchTight} on the rationality of asymptotic translation lengths of $\Mod(S)$ on $\calC(S)$.
However the key ingredient is a new algorithm for computing geodesics in $\calC(S)$:

\begin{theorem}
Fix an orientable surface $S$.
Algorithm~\ref{alg:geodesic_path} takes as input a pair of curves $a, b \in \calC(S)$ and returns a (tight) geodesic between them in polynomial time. \qed
\end{theorem}

A standard argument using curve surgery described by Hempel \cite[Lemma~2.1]{Hempel} shows that $\dC(a, b)$ is bounded above in terms of $\log(\intersection(a, b))$.
Unfortunately, $\calC(S)$ is locally infinite and so, even with this bound on the depth needed to search, na\"{\i}ve pathfinding algorithms fail immediately.

To avoid this, Algorithm~\ref{alg:geodesic_path} proceeds by first constructing a quasiconvex subset $U \subseteq \calC(S)$ which contains $a$ and $b$ (Algorithm~\ref{alg:quasiconvex}).
We then compute all tight paths between each pair of curves in $U$ of length at most $L$ using an algorithm of the second author (Algorithm~\ref{alg:tight_filling_paths}).
Tight paths were first introduced (as tight filling multipaths) in \cite{WebbCombinatorics}.
The quasiconvexity of our subset ensures that this collection of (multi)curves contains a (tight) geodesic from $a$ to $b$ (Theorem~\ref{thrm:tight_neighbourhood}).
Thus any polynomial-time geodesic finding algorithm for finite graphs can be used to find a geodesic from $a$ to $b$ through this subset.

The fact that we only ever need to construct tight paths of bounded length is critical.
Constructing tight paths becomes extremely expensive as these paths become longer, and so we avoid the high cost of computing a tight path directly from $a$ to $b$.
As a consequence, our algorithm runs exponentially faster than the geodesic finding algorithms of Birman--Margalit--Menasco \cite{BirmanMargalitMenasco}, Leasure \cite[Section~3]{LeasureThesis}, Shackleton \cite[Theorem~1.5]{Shackleton}, Watanabe \cite{WatanabeAlgorithm} and the second author \cite[Theorem~4.7]{WebbCombinatorics}.

Again we describe how we represent (multi)curves and our notion of complexity in detail in Section~\ref{sub:complexity}.
However a key property of our complexity is that if we fix an ideal triangulation $\calT$ of $S$ then the complexity of $c$ is coarsely equivalent to $\log(\intersection(c, \calT))$.
Algorithm~\ref{alg:nt_type} relies heavily on applying a large power of the given mapping class to a curve and a priori $\intersection(f(c), \calT)$ grows exponentially with the word length of $f$.
Hence this additional speed is essential for Theorem~\ref{thrm:nt_type} and Theorem~\ref{thrm:canonical}.

\subsection{Set-up}

In this paper we give an explicit proof for a fixed \emph{punctured} surface $S = S_{g, p}$ of genus $g$ with $p$ punctures, that is, when $p > 0$.
This enables us to use a polynomial-time algorithm of the authors to find the minimal position of a pair of multicurves \cite{BellWebbSimplificationApplications}.
All the results in this paper will hold for closed surfaces, when $p = 0$, by introducing an additional marked point as an artificial puncture.
To obtain minimal position for a pair of multicurves on such a surface, we isotope a curve across this marked point and into a normal form whenever necessary.
We can achieve this in polynomial time by applying Plandowski's algorithm \cite[Section~3]{Plandowski} to the corresponding cut sequence of the curve.

Throughout let $\xi = \xi(S) \defeq 3g - 3 + p$.
We will assume that $\xi \geq 2$.
In all cases excluded by this requirement the problems that we will consider have trivial or well-known solutions \cite{BeardonHockmanShort}.
Furthermore, let $\zeta = \zeta(S) \defeq -3 \chi(S)$.
We note that $\zeta \geq 3$ and so $S$ can be decomposed into an (ideal) triangulation.
Any such triangulation of $S$ has exactly $\zeta$ edges.

In order to be able to describe curves and mapping classes combinatorially, we fix one such triangulation $\calT_0$ of $S$ with edges $e_1, \ldots, e_\zeta$.

\subsection{Complexity}
\label{sub:complexity}

To define the complexities of our algorithms we recall two standard complexity theory definitions.
First, if $x$ is an integer then we define its complexity $||x||$ to be the \emph{bit-size} of $x$, that is, $||x|| \defeq \log(|x|)$.
Changing the base of the logarithm used only affects the complexity by a multiplicative factor and so will not impact the asymptotic running time of the algorithms we present.
However, it will be convenient take $\log \defeq \log_2$ throughout.
Second, if $v$ is an integer vector then we define its complexity $||v||$ to be the sum of the complexities of its entries.

Our algorithms manipulate curves and mapping classes.
We use our fixed triangulation $\calT_0$ to enable us to describe these combinatorially.

First, $\calT_0$ enables us to represent a (multi)curve or (multi)arc $a$ by its \emph{edge vector}:
\[ \calT_0(a) \defeq \left(\begin{array}{c}
\intersection(a, e_1) \\
\vdots \\
\intersection(a, e_\zeta)
\end{array}\right)  
\]
We use a slightly modified version of geometric intersection number to handle multiarcs \cite[Definition~5.1.2]{BellThesis}.
If $a$ contains $k$ copies of the edge $e_i$ then we define
\[ \intersection(a, e_i) \defeq -k. \]
We define the \emph{complexity} of $a$ to be the complexity of its corresponding edge vector, that is,
\[ ||a|| \defeq ||\calT_0(a)||. \]
Unless otherwise stated, whenever an algorithm in this paper requires a multicurve or multiarc for input, output or as an intermediate stage of a calculation we will assume that it is given via its edge vector.

Second, recall that the \emph{(ordered) flip graph} \cite[Section~2.2]{BellThesis} $\calF = \calF(S)$ is the graph with a vertex for each triangulation (where the edges are ordered).
Triangulations $\calT, \calT' \in \calF$ are connected via an edge (of length one) if and only if they differ by a \emph{flip}, as shown in Figure~\ref{fig:flip}, or a re-ordering.

\begin{figure}[ht]
\centering
\begin{tikzpicture}[scale=2,thick]

\node (rect) at (-1.5, 0) [draw,minimum width=3cm,minimum height=3cm] {};
\node (rect2) at (1, 0) [draw,minimum width=3cm,minimum height=3cm] {};

\draw (rect.south west) -- node [above, anchor=south] {$e$} (rect.north east);
\draw (rect2.north west) -- (rect2.south east);

\node [dot] at (rect.north west) {};
\node [dot] at (rect.north east) {};
\node [dot] at (rect.south west) {};
\node [dot] at (rect.south east) {};

\node [dot] at (rect2.north west) {};
\node [dot] at (rect2.north east) {};
\node [dot] at (rect2.south west) {};
\node [dot] at (rect2.south east) {};

\draw [thick,->] ($(rect.east)!0.25!(rect2.west)$) -- node[above] {Flip} ($(rect.east)!0.75!(rect2.west)$);

\node at (rect.south) [below,anchor=north] {$\calT$};
\node at (rect2.south) [below,anchor=north] {$\calT'$};

\end{tikzpicture}
\caption{Flipping the edge $e$ of $\calT$.}
\label{fig:flip}
\end{figure}
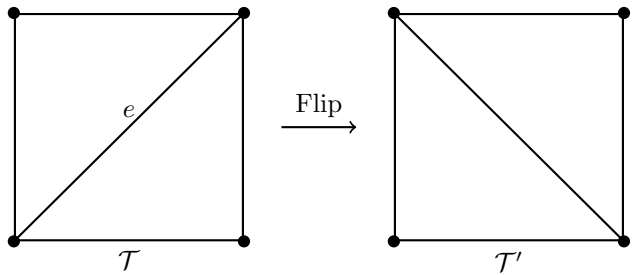

The flip graph is connected \cite[Page~190]{HatcherTriangulations} and is quasi-isometric to the mapping class group $\Mod(S)$ \cite[Proposition~I.8.19]{BridsonHaefliger}.
Thus we represent a mapping class $f$ via a path from $\calT_0$ to $f(\calT_0)$.
We define the \emph{complexity} of $f$, denoted $||f||$, to be the length of the given path.
For any word in a fixed generating set we can build such a path whose length is coarsely the word length.

Once more, unless otherwise stated, whenever an algorithm in this paper requires a mapping class for input, output or as an intermediate stage of a calculation we will assume that it is given via such a path.

Several calculations are straightforward with these choices of data structures.
For example, there are polynomial-time algorithms known for computing:
\begin{itemize}
\item whether a vector corresponds to a (multi)curve \cite[Lemma~2.1.2]{BellThesis},
\item the intersection number between two curves \cite[Corollary~3.5]{BellWebbSimplificationApplications},
\item whether two curves fill \cite[Section~4.3]{BellWebbSimplificationApplications},
\item whether two curves are equal \cite[Page~11]{BellThesis},
\item the components and multiplicities of a multicurve \cite{AgolHassThurston} \cite[Section~4]{BellSimplifying} \cite[Theorem~6.9, Corollary~6.10]{EricksonNayyeri},
\item the image of a curve under a mapping class \cite[Lemma~2.2.2]{BellThesis} \cite[Page~30]{MosherAutomatic},
\item the composition of two mapping classes, and
\item the order of a mapping class \cite[Theorem~7.5]{FarbMargalit} \cite{MosherAutomatic}.
\end{itemize}
These algorithms enable us to make several simplifying assumptions.
For instance, throughout this paper we will assume that any multicurve or multiarc has had its inessential and parallel components removed.
Thus multicurves correspond to cliques in $\calC(S)$.
This will be particularly important for boundaries of submanifolds, which will appear frequently in the form $\partial N(c \cup c')$, that is, the boundary of a regular neighbourhood of $c$ and $c'$ after they have been isotoped into minimal position.
For example, see Figure~\ref{fig:boundary}.

\begin{figure}[ht]
\centering
\begin{tikzpicture}[scale=1.35,thick]

\draw [red, dotted] (-1,-0.15) to [out=180,in=180] (-1,-0.5);
\draw [red, dotted] (-1,-0.0) to [out=180,in=180] (-1,0.5);
\draw [blue, dotted] (-2,0) to [out=270,in=270] (-1.28,-0.09);

\path [pattern=north west lines, pattern color=red]
	(0,0.5-0.1) to [out=0,in=180]
	(1,0.5) to [out=0,in=90]
	(2,0) to [out=270,in=0]
	(1,-0.5) to [out=180,in=0]
	(0,-0.5+0.1) to [out=180,in=0]
	(-1,-0.5) to [out=0,in=0]
	(-1,-0.15) to [out=90,in=270]
	(-1,-0.0) to [out=0,in=0]
	(-1,0.5) to [out=0,in=180]
	(0,0.5-0.1);

\draw [fill, white] (1.75,0) circle (5pt);
\draw [red] (1.75,0) circle (5pt);

\draw [fill, white] (0,0) circle (5pt);
\draw [red] (0,0) circle (5pt);

\draw [fill, white] (-0.72,-0.09) to [out=155,in=25]
	(-1.28,-0.09) to [out=340,in=200]
	(-0.72,-0.09);
\draw [fill, white] (1.28,-0.09) to [out=155,in=25]
	(0.72,-0.09) to [out=340,in=200]
	(1.28,-0.09);

\draw
	(0,0.5-0.1) to [out=0,in=180]
	(1,0.5) to [out=0,in=90]
	(2,0) to [out=270,in=0]
	(1,-0.5) to [out=180,in=0]
	(0,-0.5+0.1) to [out=180,in=0]
	(-1,-0.5) to [out=180,in=270]
	(-2,0) to [out=90,in=180]
	(-1,0.5) to [out=0,in=180]
	(0,0.5-0.1);

\draw (-1.5,0) to [out=-30,in=180+30] (-0.5,0);
\draw (-1.3,-0.1) to [out=30,in=180-30] (-0.7,-0.1);
\draw (0.5,0) to [out=-30,in=180+30] (1.5,0);
\draw (0.7,-0.1) to [out=30,in=180-30] (1.3,-0.1);

\draw [red] (-1,-0.15) to [out=0,in=0] (-1,-0.5);
\draw [red] (-1,-0.0) to [out=0,in=0] (-1,0.5);
\draw [blue] (-2,0) to [out=90,in=90] (-1.28,-0.09);

\node [blue] at (-2.5,0) {$\partial Y$};

\node [dot] at (1.75,0) {};
\node at (0,-0.75) {$\textcolor{red}{Y} \subseteq S$};

\draw [use as bounding box, draw=none] (-2,-0.75) rectangle (2,0.5);  

\end{tikzpicture}
\caption{As peripheral, null-homotopic and parallel components are removed, $\partial Y$ is a multicurve with a single component.}
\label{fig:boundary}
\end{figure}
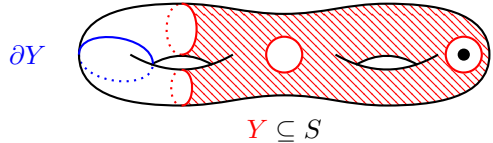

We also highlight the following inequalities which will be used throughout:
\begin{itemize}
\item $||f(c)|| \leq ||f|| + ||c||$,
\item $||f \circ g|| \leq ||f|| + ||g||$ and so $||f^n|| \leq n ||f||$,
\item $||\partial N(c \cup c')|| \leq \max(||c||, ||c'||) + 2$, and
\item if $c' \subseteq c$ then $||c'|| \leq ||c||$.
\end{itemize}

\section{Quasiconvex subsets}

In this section we give a polynomial-time algorithm to compute a $K'$--quasiconvex subset of $\calC(S)$ containing a pair of curves $a$ and $b$.
(See Algorithm~\ref{alg:quasiconvex}.)
The constant $K'$ is in fact universal.

This construction is important for Section~\ref{sec:geodesics}.
It provides a guide for our construction of a geodesic in $\calC(S)$ from $a$ to $b$.

We begin by describing an elementary construction of a $K$--quasiconvex subset of $\calC(S)$ containing $a$ and $b$.
Unfortunately, the subset produced by this na\"{\i}ve procedure may contain exponentially many curves in terms of the complexity of $a$ and $b$.
We describe a refinement which enables us to skip large sections of this subset which contribute almost nothing.
We show how to construct this subset which is still quasiconvex but now only contains polynomially many curves.

One key observation is that without loss of generality we may assume that
\[ a \in \short(\calT_0) \defeq \{ \textrm{curves that meet each edge of $\calT_0$ at most twice} \}. \]
This follows from the fact that we can simplify $\calT_0$ with respect to $a$ to obtain an $a$--minimal triangulation $\calT_1$ \cite{BellSimplifying}.
On this new triangulation $a \in \short(\calT_1)$ and $||\calT_1(b)|| \leq ||b|| + O(||a||^2)$ \cite[Corollary~3.5]{BellWebbSimplificationApplications}.
We can interchange between $\calT_0$ and $\calT_1$ representations of curves in polynomial time and at a complexity cost of only $O(||a||^2)$.
Hence it is sufficient to produce a $K$--quasiconvex subset containing $a$ and $b$ in the coordinates given by $\calT_1$ as we can then convert the curves of this subset to the coordinates given by $\calT_0$ afterwards.

Let $\tau_1$ be the (measured) train track obtained by retracting $b$ in each triangle of $\calT_0$ to one of the four possibilities shown in Figure~\ref{fig:elementary_train_track}.
We refer to this as the \emph{elementary train track} of $b$.

\begin{figure}[ht]
\centering
\begin{tabular}{cc}
	\begin{subfigure}[t]{0.4\textwidth}
		\centering
		\begin{tikzpicture}[scale=2,thick]
			\coordinate (A) at (210:0.577);  
			\coordinate (B) at (330:0.577);
			\coordinate (C) at (90:0.577);
			\coordinate (A2) at (210:0.3);
			\coordinate (B2) at (330:0.3);
			\coordinate (C2) at (90:0.3);
			\draw (A) -- node [below] {$e_k$} (B) -- node [right] {$e_j$} (C) -- node [left] {$e_i$} (A);
		\end{tikzpicture}
		\caption*{$\intersection(b, e_i) = \intersection(b, e_j) = \intersection(b, e_k) = 0$.}
	\end{subfigure} &
	\begin{subfigure}[t]{0.4\textwidth}
		\centering
		\begin{tikzpicture}[scale=2,thick]
			\coordinate (A) at (210:0.577);  
			\coordinate (B) at (330:0.577);
			\coordinate (C) at (90:0.577);
			\draw (A) -- node [below] {$e_k$} (B) -- node [right] {$e_j$} (C) -- node [left] {$e_i$} (A);
			\coordinate (A2) at (30:0.2);
			\coordinate (B2) at (150:0.2);
			\coordinate (C2) at (270:0.2);
			\draw [red] ($(A)!0.5!(C)$) to (B2) to [out=330, in=210] (A2) to ($(B)!0.5!(C)$);
		\end{tikzpicture}
		\caption*{$\intersection(b, e_i) = \intersection(b, e_j)$ and $\intersection(b, e_k) = 0$.}
	\end{subfigure} \\ \\
	\begin{subfigure}[t]{0.4\textwidth}
		\centering
		\begin{tikzpicture}[scale=2,thick]
			\coordinate (A) at (210:0.577);  
			\coordinate (B) at (330:0.577);
			\coordinate (C) at (90:0.577);
			\draw (A) -- node [below] {$e_k$} (B) -- node [right] {$e_j$} (C) -- node [left] {$e_i$} (A);
			\coordinate (A2) at (30:0.2);
			\coordinate (B2) at (150:0.2);
			\coordinate (C2) at (270:0.2);
			\draw (A) -- node [below] {$e_k$} (B) -- node [right] {$e_j$} (C) -- node [left] {$e_i$} (A);
			\draw [red] ($(A)!0.5!(C)$) to (B2) to [out=330, in=90] (C2) to ($(A)!0.5!(B)$);
			\draw [red] ($(B)!0.5!(C)$) to (A2) to [out=210, in=90] (C2) to ($(A)!0.5!(B)$);
		\end{tikzpicture}
		\caption*{$\intersection(b, e_i) + \intersection(b, e_j) = \intersection(b, e_k)$.}
	\end{subfigure} &
	\begin{subfigure}[t]{0.4\textwidth}
		\centering
		\begin{tikzpicture}[scale=2,thick]
			\coordinate (A) at (210:0.577);  
			\coordinate (B) at (330:0.577);
			\coordinate (C) at (90:0.577);
			\draw (A) -- node [below] {$e_k$} (B) -- node [right] {$e_j$} (C) -- node [left] {$e_i$} (A);
			\coordinate (A2) at (30:0.2);
			\coordinate (B2) at (150:0.2);
			\coordinate (C2) at (270:0.2);
			\draw [red] ($(A)!0.5!(C)$) to (B2) to [out=330, in=90] (C2) to ($(A)!0.5!(B)$);
			\draw [red] ($(B)!0.5!(C)$) to (A2) to [out=210, in=90] (C2) to ($(A)!0.5!(B)$);
			\draw [red] ($(A)!0.5!(C)$) to (B2) to [out=330, in=210] (A2) to ($(B)!0.5!(C)$);
		\end{tikzpicture}
		\caption*{Otherwise.}
	\end{subfigure}
\end{tabular}
\caption{The different possibilities for $\tau_1$ corresponding to $b$. Switches are in the interior of the triangles and so $\tau_1$ is trivalent.}
\label{fig:elementary_train_track}
\end{figure}
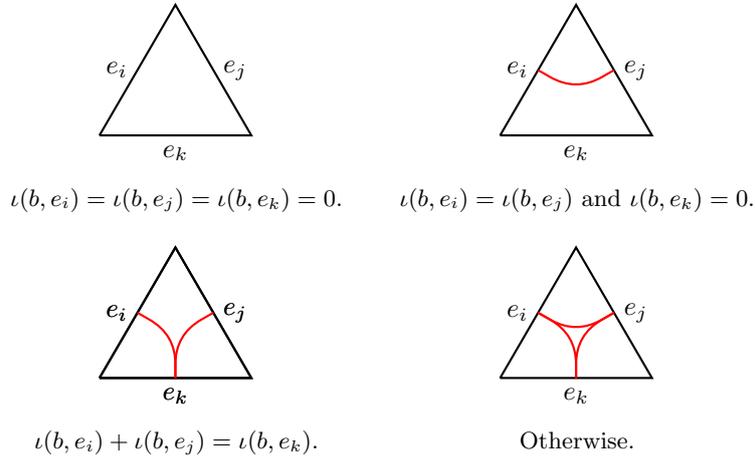

\begin{definition}
A \emph{selected train track} is a (measured) train track $\tau$ with a preferred switch $s$, written as the pair $(\tau, s)$.
We use the pair $(\tau, \emptyset)$ if $\tau$ does not have any switches.
\end{definition}

We follow the curve tracing construction of Erickson and Nayyeri in the setting of train tracks \cite{EricksonNayyeri}.
Suppose that $(\tau_i, s_i)$ is a selected train track where $s_i \neq \emptyset$.
Let $b_i$ be the incoming branch of $s_i$.
\begin{itemize}
\item If $b_i$ is a mixed branch then we shift $\tau_i$ along $b_i$ to obtain $\tau_{i+1}$.
We then set $s_{i+1}$ as shown in Figure~\ref{fig:shift}.
\item If $b_i$ is a large branch then we split $\tau_i$ along $b_i$ to obtain $\tau_{i+1}$.
If the splitting is a left or right splitting then we set $s_{i+1}$ as shown in Figure~\ref{fig:split}.
If the splitting is a central splitting then we set $s_{i+1}$ to be any switch of $\tau_{i+1}$ (or $\emptyset$ if $\tau_{i+1}$ does not have any switches).
\end{itemize}
In either case, we obtain a new selected train track $(\tau_{i+1}, s_{i+1})$.
We apply this process repeatedly, starting with $(\tau_1, s_1)$ where $s_1$ is a switch of $\tau_1$, to obtain a sequence of selected train tracks.

\begin{figure}[ht]
\centering
\begin{tikzpicture}[scale=2,thick]

\draw (0.5, 0) -- (2, 0);
\draw (0.5,0.5) to [out=-60,in=180] (1,0) node [dot,label=below:{$s_{i}$}] {};
\draw (1,-0.5) to [out=60,in=180] (1.5,0);
\node at (1.25,-0.75) {$\tau_{i}$};

\begin{scope}[shift={(2.5,0)}]
\draw (0.5, 0) -- (2, 0);
\draw (1,0.5) to [out=-60,in=180] (1.5,0) node [dot,label=below:{$s_{i+1}$}] {};
\draw (0.5,-0.5) to [out=60,in=180] (1,0);
\node at (1.25,-0.75) {$\tau_{i+1}$};
\end{scope}

\draw [->] (2.25,0) -- node [above] {Shift} (2.75,0);

\end{tikzpicture}
\caption{Shifting when the incoming branch is mixed.}
\label{fig:shift}
\end{figure}
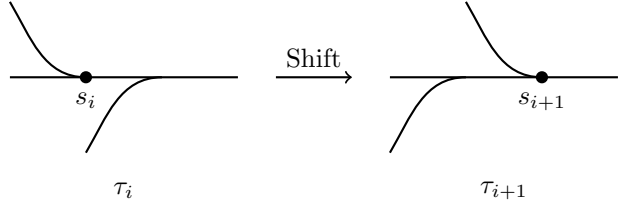

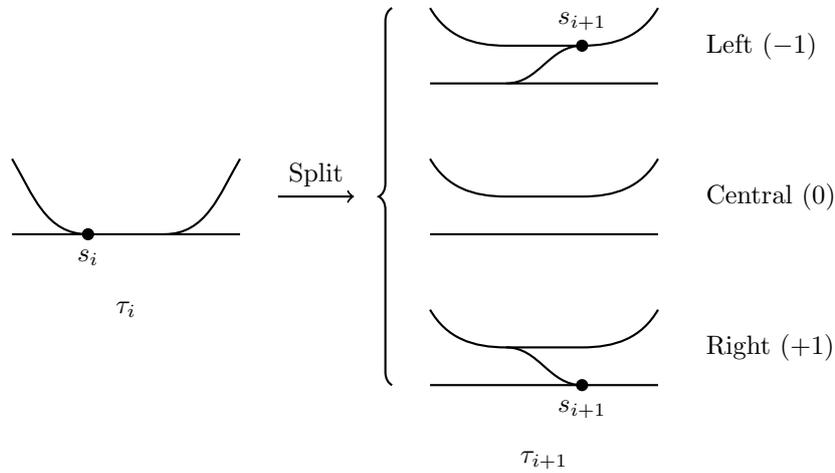
\begin{figure}[ht]
\centering
\begin{tikzpicture}[scale=2,thick]

\draw (0.5, 0) -- (2, 0);
\draw (0.5,0.5) to [out=-60,in=180] (1,0) node [dot,label=below:{$s_{i}$}] {};
\draw (2,0.5) to [out=240,in=0] (1.5,0);
\node at (1.25,-0.5) {$\tau_{i}$};

\begin{scope}[shift={(2.75,-1)}]
\draw (0.5, 0) -- (2, 0);
\draw (0.5,0.5) to [out=-60,in=180] (1,0.25) to (1.5,0.25) to [out=0,in=240] (2,0.5);
\draw (1,0.25) to [out=0,in=180] (1.5,0) node [dot,label=below:{$s_{i+1}$}] {};
\node [anchor=west] at (2.25,0.25) {Right ($+1$)};
\end{scope}

\begin{scope}[shift={(2.75,0)}]
\draw (0.5, 0) -- (2, 0);
\draw (0.5,0.5) to [out=-60,in=180] (1,0.25) to (1.5,0.25) to [out=0,in=240] (2,0.5);
\node [anchor=west] at (2.25,0.25) {Central ($0$)};
\end{scope}

\begin{scope}[shift={(2.75,1)}]
\draw (0.5, 0) -- (2, 0);
\draw (0.5,0.5) to [out=-60,in=180] (1,0.25) to (1.5,0.25) to [out=0,in=240] (2,0.5);
\draw (1,0) to [out=0,in=180] (1.5,0.25) node [dot,label=above:{$s_{i+1}$}] {};
\node [anchor=west] at (2.25,0.25) {Left ($-1$)};
\end{scope}

\node at (4,-1.5) {$\tau_{i+1}$};
\draw [->] (2.25,0.25) -- node [above] {Split} (2.75,0.25);
\draw [decorate,decoration={brace,amplitude=5pt,mirror}] (3,1.5) -- (3,-1);

\end{tikzpicture}
\caption{Splitting when the incoming branch is large.}
\label{fig:split}
\end{figure}

Now any train track on $S$ has at most $3 \zeta$ branches.
Therefore this sequence contains at most $\zeta$ central splits and we cannot ever perform more than $6 \zeta$ shifts in a row.
Furthermore, the sum of the weights on the large branches stays constant when a shift is performed and decreases when a split is performed.
Hence this sequence of selected train tracks eventually reaches $(\tau_n, s_n) = (b, \emptyset)$ at which point this sequence stops.

\begin{definition}
For a (measured) train track $\tau$, let $\short(\tau) \subseteq \calC(S)$ denote the set of curves carried by $\tau$ that run over each branch at most twice.
\end{definition}

For each $(\tau_i, s_i)$ in this sequence of selected train tracks we now choose $c_i \in \short(\tau_i)$.
We note that $\tau_n = b$ and so $\short(\tau_n) = \{ b \}$ and $c_n = b$.

\begin{remark}
We note that if $c$ is carried by $\tau_i$ then it is also carried by $\tau_1$.
This is denoted $\tau_i \carriedby \tau_1$.
Furthermore, by reversing the sequence of shifts and splits we can compute the transverse measure assigned to $\tau_1$ by $c$.
Thus we can compute the edge vector $\calT_0(c_i)$ for each chosen $c_i \in \short(\tau_i)$.
\end{remark}

We do not attempt to track how each $\tau_i$ is embedded into $S$ as this requires an exponential amount of information.
That is, we consider each $\tau_i$ as an abstract graph with decorations recording the incoming, left outgoing and right outgoing branches at each switch.

\begin{lemma}
\label{lem:quasiconvex_bounds}
If $a \in \short(\calT_0)$ then $||c_i|| \leq ||b|| + \zeta$.
More generally, $||c_i|| \leq ||b|| + O(||a||^2)$.
\end{lemma}

\begin{proof}
We have that $c_i$ runs over each branch of $\tau_i$ at most twice, while $b$ runs over each branch of $\tau_i$ at least once.
Therefore, if $a \in \short(\calT_0)$ then the measure assigned to each branch of $\tau_1$ by $c_i$ is at most double the measure assigned by $b$.
Hence $\intersection(c_i, e_i) \leq 2 \intersection(b, e_i)$ for each edge $e_i$ of $\calT_0$ and so the result holds by taking logarithms.

More generally, if $a \notin \short(\calT_0)$ then we can convert the $\calT_1(c_i)$ coordinates back to $\calT_0(c_i)$ at a cost of $O(||a||^2)$.
Hence the general bound holds also.
\end{proof}

\begin{proposition}[{\cite[Theorem~1.3]{Aougab}}]
\label{prop:short_distance}
Whenever $\zeta$ is sufficiently large, if $\intersection(x, y) \leq 24 \zeta$ then $\dC(x, y) \leq 3$.
Hence there is a universal constant $T$ such that if $\intersection(x, y) \leq 24 \zeta$ then $\dC(x, y) \leq T$. \qed
\end{proposition}

\begin{remark}
Now for any two curves $x, y \in \short(\tau)$ we have that $\intersection(x, y) \leq 12 \zeta$.
As a result, it will not matter which curve in $\short(\tau_i)$ we choose.
\end{remark}

\begin{lemma}
For consecutive terms in our sequence, $\dC(c_i, c_{i+1}) \leq T$.
\end{lemma}

\begin{proof}
If $(\tau_{i+1}, s_{i+1})$ was obtained from $(\tau_i, s_i)$ by shifting then $\short(\tau_{i+1}) = \short(\tau_i)$ and so $\dC(c_i, c_{i+1}) \leq T$ by the preceding remark.

On the other hand if $(\tau_{i+1}, s_{i+1})$ was obtained from $(\tau_i, s_i)$ by splitting then $c_{i+1}$ is carried by $\tau_i$ and runs over each branch at most four times.
Therefore $\intersection(c_i, c_{i+1}) \leq 24 \zeta$ and so $\dC(c_i, c_{i+1}) \leq T$ by Proposition~\ref{prop:short_distance}.
\end{proof}

\begin{theorem}[{\cite[Theorem~1.3]{MasurMinskyQuasiconvexity}}]
\label{thrm:quasiconvex}
For every $M > 0$ there is a constant $K = K(S, M)$ such that if $\tau_n \carriedby \cdots \carriedby \tau_1$ is a sequence of nested train tracks and $c_i \in \short(\tau_i)$ and $\dC(c_i, c_{i+1}) \leq M$ then $\{c_1, \ldots, c_n\}$ is a $K$--quasiconvex subset of $\calC(S)$. \qed
\end{theorem}

Hence, as $T$ is universal, there is a constant $K = K(S)$ such that the subset $\{c_1, \ldots, c_m\}$ chosen above is $K$--quasiconvex in $\calC(S)$.
In \cite[Corollary~3.6]{Hamenstaedt}, Hamenst\"{a}dt showed that in fact $K$ may also be chosen to be a universal constant.

Unfortunately, we are only guaranteed that every $6 \zeta$ steps the sum of the weights on the large branches decreases by at least one.
As this starts at $\intersection(b, \calT_0)$, this na\"{\i}ve sequence of selected train tracks can be exponentially long in $||b||$.
Hence we cannot compute all of the terms in this sequence in polynomial time.
To overcome this we use a key idea of the construction of Agol--Hass--Thurston \cite[Section~4]{AgolHassThurston}, summarised succinctly by Erickson and Nayyeri \cite{EricksonNayyeri} as:
\begin{quote}
If you perform more than $6 \zeta$ splits with the same sign then you must be repeatedly unwinding along a curve (see Figure~\ref{fig:twist}).
\end{quote}
We use the convention that left, central and right splittings have \emph{sign} $-1$, $0$ and $+1$ respectively, as shown in Figure~\ref{fig:split}.

\begin{figure}[ht]
\centering
\begin{tikzpicture}[scale=2,thick]

\draw (0.5,0.5) to [out=-60,in=180] (1,0) node [dot,label=above:\contour*{white}{{$s_i$}}] {};
\foreach \x in {0.3,0.4,1.8,1.9,2} {\draw (\x,0) to [out=180,in=80] (\x-0.3,-0.5);}
\foreach \x in {0.5,0.7,1.1} {\draw (\x,0) to [out=0,in=100] (\x+0.3,-0.5);}

\draw [>->] (0,0) -- (2,0);
\draw [>->,red] (0,0.25) -- (2,0.25) node [right] {$c$};
\node at (1,-0.75) {$\tau_{i}$};

\begin{scope}[shift={(3,0)}]
\foreach \y in {0.15,0.2,0.25,0.3,0.35} {\draw (0,\y) to [out=0,in=180] (2,\y-0.05);}
\draw (0,0.1) to [out=0,in=180] (0.75,0.1) to [out=0,in=180] (1,0) node [dot,label=above:\contour*{white}{{$s_j$}}] {};
\draw (0.5,0.5) to [out=-60,in=180] (0.75,0.4) to [out=0,in=180] (2,0.35);
\foreach \x in {0.3,0.4,1.8,1.9,2} {\draw (\x,0) to [out=180,in=80] (\x-0.3,-0.5);}
\foreach \x in {0.5,0.7,1.1} {\draw (\x,0) to [out=0,in=100] (\x+0.3,-0.5);}

\draw [>->] (0,0) -- (2,0);
\node at (1,-0.75) {$\tau_{j} = T_c^k(\tau_i)$};
\end{scope}

\draw [->] (2.25,0) -- node [above] {Twist} (2.75,0);

\end{tikzpicture}
\caption{Twisting to accelerate through a (left) unwinding.}
\label{fig:twist}
\end{figure}
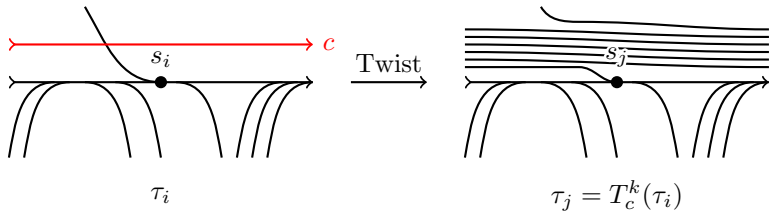

Therefore, if we reach $(\tau_i, s_i)$ having just performed a split and the last $6 \zeta$ splits all had the same sign then we accelerate through the unwinding process that is about to occur.
We achieve this by computing the curve $c$ we are unwinding along, which is in $\short(\tau_i)$, and the largest power of a Dehn twist along $c$ that we can apply to $\tau_i$ so that it still carries $b$.
This power is simply $\lfloor m / w \rfloor$ where $w$ is the weight on the branch of $\tau_i$ merging into $c$ and $m$ is the minimal ratio between the weight on a branch along $c$ and the number of times that $c$ runs over it \cite[Section~3.2]{EricksonNayyeri}.
See Figure~\ref{fig:twist}.
Thus we jump directly from $(\tau_i, s_i)$ to $(\tau_j, s_j)$ where, potentially, $j \gg i$.

This gives us a subsequence $(\tau_{i_1}, s_{i_1}), \ldots, (\tau_{i_m}, s_{i_m})$.
Erickson and Nayyeri showed that this subsequence contains only $O(||b||)$ curves \cite[Section~4]{EricksonNayyeri}.
However, in the corresponding subsequence of short curves $c_{i_1}, \ldots, c_{i_m}$, the potentially large jumps actually still do not move us far in $\calC(S)$:

\begin{proposition}
Consecutive terms in our subsequence are at most distance $T + 6$ apart in $\calC(S)$.
\end{proposition}

\begin{proof}
Suppose that we jump from $\tau_i$ to $\tau_{j} = T_c^k(\tau_i)$.
The curve $c$ meets $\tau_i$ at most once.
Hence, as $c_i \in \short(\tau_i)$ runs over each branch at most twice, $\intersection(c_i, c) \leq 2$ and so $\dC(c_i, c) \leq 3$.
Therefore
\begin{align*}
\dC(c_i, T_c^k(c_i)) &\leq \dC(c_i, c) + \dC(c, T_c^k(c_i)) \\
 &= \dC(c_i, c) + \dC(T_c^k(c), T_c^k(c_i)) \\
 &= 2 \dC(c_i, c) \\
 &\leq 6
\end{align*}
Now $T_c^k(c_i) \in \short(\tau_j)$ and $\diam(\short(\tau_j)) \leq T$ by Proposition~\ref{prop:short_distance}.
Hence
\[ \dC(c_i, c_j) \leq T + 6. \]

As before $\dC(c_i, c_{i+1}) \leq T$.
Hence the distance between consecutive curves in this subsequence is bounded above by $T + 6$.
\end{proof}

Therefore once again by Theorem~\ref{thrm:quasiconvex}, we have that the $\{c_{i_1}, \ldots, c_{i_m}\}$ is a quasiconvex subset of $\calC(S)$ which contains $b$.
Furthermore, since we only jump forwards after performing $6 \zeta$ splits with the same sign in a row we have that $c_{i_1} = c_1$ and $c_{i_m} = c_n$.
Hence this subset contains $b$.

Finally, for ease of notation let $c_{i_0} = c_0 \defeq a$.
Now although $c_0$ may not be in $\short(\tau_1)$ we still have that $c_0, c_1 \in \short(\calT_0)$.
It follows from this that $\intersection(c_0, c_1) \leq 4 \zeta$ and so $\dC(c_0, c_1) \leq T$ by Proposition~\ref{prop:short_distance}.

The curve graph is $\delta$--hyperbolic \cite[Theorem~1.1]{MasurMinskyI}, that is, all triangles are $\delta$--slim.
In fact $\delta$ is a universal constant and has been explicitly bounded \cite[Theorem~1.1]{Aougab} \cite[Theorem~1.1]{BowditchUniform} \cite[Corollary~7.1]{ClayRafiSchleimer} \cite[Theorem~1.1]{HenselPrzytyckiWebb}.
Hence if we add a single point to a quasiconvex subset of $\calC(S)$ then the resulting subset is also quasiconvex.

Therefore $\{c_{i_0}, c_{i_1}, \ldots, c_{i_m}\}$ is a $K'$--quasiconvex subset of $\calC(S)$ containing $a$ and $b$.
As $K'$ depends only on $K$, $T$ and $\delta$, it too is a universal constant.

\begin{corollary}
Algorithm~\ref{alg:quasiconvex} computes a $K'$--quasiconvex subset of $\calC(S)$ containing $a$ and $b$ in polynomial time. \qed
\end{corollary}

\begin{algorithm}[ht!]
\caption{Quasiconvex subsets in $\calC(S)$}
\label{alg:quasiconvex}
\begin{algorithmic}
\Require{Curves $a$ and $b$ where $a \in \short(\calT_0)$.}
\Ensure{A $K'$--quasiconvex subset of $\calC(S)$ containing $a$ and $b$.}
	\State \Comment{We drop the subsequence notation; so $\tau_5$ here corresponds to $\tau_{i_5}$ above.}
	\State $\tau_1 \gets$ the elementary train track associated to $b$
	\IIf {$\tau_1$ does not have any switches}
		$s_1 \gets \emptyset$
	\EElse
		$s_1 \gets$ a switch of $\tau_1$
	\EndIIf
	\State current sign $\gets 0$ \Comment{The type of splitting we have been performing.}
	\State count $\gets 0$ \Comment{The number of these splits we have performed.}
	\State $i \gets$ 1
	\While {$s_i \neq \emptyset$} \Comment{$\tau_i \neq b$. This loop only runs $O(||b||)$ times.}
		\State $b_i \gets$ incoming branch of $s_i$
		\If {$b_i$ is mixed}
			\State $(\tau_{i+1}, s_{i+1}) \gets (\tau_i, s_i)$ shifted along $b_i$ \Comment{See Figure~\ref{fig:shift}.}
		\Else \Comment{$b_i$ is large.}
			\State $(\tau_{i+1}, s_{i+1}) \gets (\tau_i, s_i)$ split along $b_i$ \Comment{See Figure~\ref{fig:split}.}
			\State sign $\gets$ sign of split performed \Comment{Either $-1$, $0$ or $+1$.}
			\If {sign $=$ current sign}
				\State count $\gets$ count + 1
			\Else \Comment{Start counting again.}
				\State current sign $\gets$ sign
				\State count $\gets 1$
			\EndIf
		\EndIf
		\State $i \gets i+1$
		\If {count $> 6 \zeta$} \Comment{Accelerate by twisting. Note sign $\neq 0$ here.}
			\State $c \gets$ the curve we have been splitting along
			\State $w \gets$ the weight on the outgoing branch of $s_i$ that is not in $c$
			\State $m \gets \min \{\textrm{weight on branch} / \textrm{number of times it is transversed by $c$} \}$
			\State $k \gets \textrm{sign} \cdot \lfloor m / w \rfloor$ \Comment{Multiply by sign to twist in the correct direction.}
			\State $(\tau_{i+1}, s_{i+1}) \gets T_c^k((\tau_i, s_i))$ \Comment{See Figure~\ref{fig:twist}.}
			\State $i \gets i+1$
			\State count $\gets 0$
		\EndIf
	\EndWhile
	\State $c_0 \gets a$
	\For {$j \gets 1, \ldots, i$}
		\State $c_j \gets$ a curve in $\short(\tau_j)$
	\EndFor
	\State \Return $c_0, \ldots, c_i$
\end{algorithmic}
\end{algorithm}

\section{Geodesics}
\label{sec:geodesics}

In this section we give a polynomial-time algorithm to compute a geodesic $c_0, \ldots, c_n$ in $\calC(S)$ between a pair of curves $a$ and $b$.
(See Algorithm~\ref{alg:geodesic}.)

In particular, we note that if $c_0, \ldots, c_n$ is a geodesic from $a$ to $b$ then $d_{\calC(S)}(a, b) = n$.
Hence one immediate consequence of being able to compute a geodesic in polynomial time is that:

\begin{corollary}
Algorithm~\ref{alg:distance} computes $d_{\calC(S)}(a, b)$ in polynomial time. \qed
\end{corollary}

\begin{algorithm}[ht]
\caption{Distance in $\calC(S)$.}
\label{alg:distance}
\begin{algorithmic}
\Require{Curves $a$ and $b$.}
\Ensure{$d_{\calC(S)}(a, b)$.}
	\State $\gamma \gets$ a geodesic from $a$ to $b$ \Comment{Use Algorithm~\ref{alg:geodesic}.}
	\State \Return $|\gamma| - 1$
\end{algorithmic}
\end{algorithm}

Throughout this section we use the idea of tight geodesics, first introduced by Masur and Minsky \cite[Definition~4.2]{MasurMinskyII}.

\begin{definition}
A \emph{multipath} (of length $n$) is a sequence of multicurves $a_0, \ldots, a_n$ in which consecutive multicurves are disjoint and do not share any components.
That is, $\intersection(a_i, a_{i+1}) = 0$ and $a_i \cap a_{i+1} = \emptyset$ for each $i$.
\end{definition}

\begin{definition}[{\cite[Section~3]{WebbCombinatorics}}]
A multipath $a_0, \ldots, a_n$ is \emph{tight} if:
\begin{itemize}
\item $a_i = \partial N(a_{i-1} \cup a_{i+1})$ for each $i$, and
\item $a_i$ and $a_j$ fill $S$ whenever $|i - j| \geq 3$.
\end{itemize}
We also refer to a tight multipath as a \emph{tight path}.
\end{definition}

If $a_0, \ldots, a_n$ is a multipath then by choosing a component $c_i \subseteq a_i$ for each $i$ we obtain an \emph{induced path} $c_0, \ldots, c_n$ in $\calC(S)$.
Disjointness of $a_i$ and $a_{i+1}$ implies that $c_i$ and $c_{i+1}$ are disjoint while $a_i \cap a_{i+1} = \emptyset$ implies that $c_i \neq c_{i+1}$.

\begin{definition}
A multipath is \emph{geodesic} if all of its induced paths in $\calC(S)$ are geodesics.
Again, we refer to a tight geodesic multipath as a \emph{tight geodesic}.
\end{definition}

\subsection{Tight paths of given length}

Masur and Minsky showed that there is a tight geodesic between every pair of curves \cite[Lemma~4.5]{MasurMinskyII}.

Work of the second author \cite[Theorem~4.7]{WebbCombinatorics} provides an elementary procedure to construct $P_\ell(a, b)$, the set of all tight paths of length $\ell$ between two multicurves $a$ and $b$.
This algorithm was also discovered independently by Watanabe \cite{WatanabeTight} though not explicitly stated.
We recall this procedure and show that for fixed $\ell$ it runs in polynomial time.
(See Algorithm~\ref{alg:tight_filling_paths}.)

We start with two trivial cases:
\begin{enumerate}[(1)]
\item \label{itm:base_1} If $\ell = 1$ then $a$ and $b$ form a tight path if and only if they are disjoint and do not share any components.
Thus,
\[ P_1(a, b) = \begin{cases}
\{(a, b)\} & \textrm{if} \; \intersection(a, b) = 0 \; \textrm{and} \; a \cap b = \emptyset  \\
\emptyset & \textrm{otherwise.}
\end{cases} \]
\item \label{itm:base_2} If $\ell = 2$ then the middle term of a tight path from $a$ to $b$ must be $m \defeq \partial N(a \cup b)$.
In order to form a tight path $m$ must be non-empty and cannot share any components with $a$ or $b$, thus
\[ P_2(a, b) = \begin{cases}
\{(a, \partial N(a \cup b), b)\} & \textrm{if} \; \partial N(a \cup b) \neq \emptyset \; \textrm{and} \; \partial N(a \cup b) \cap (a \cup b) = \emptyset  \\
\emptyset & \textrm{otherwise.}
\end{cases} \]
\end{enumerate}
In \cite[Section~4.3]{BellWebbSimplificationApplications} the authors gave a polynomial-time procedure to construct $\partial N(x \cup y)$.
Therefore, all of these tests and constructions can be done in polynomial time.

Therefore we devote the remainder of this section to constructing tight paths when $\ell \geq 3$.
We proceed by induction on $\ell$, that is, we assume that we can compute $P_{\ell-1}(x, y)$ for any pair of multicurves $x$ and $y$.
Additionally, we may assume that $a$ and $b$ fill as otherwise $P_\ell(a, b) = \emptyset$ automatically.

To begin, let $S'$ denote the (possibly disconnected) surface obtained by \emph{crushing} $S$ along $a$ \cite[Definition~3.4.1]{BellThesis}.
For example, see Figure~\ref{fig:crush}.

\begin{figure}[ht]
\centering
\begin{tikzpicture}[scale=1.35,thick]

\begin{scope}[shift={(-2.5,0)}]
\draw [red, dotted] (-1,-0.15) to [out=180,in=180] (-1,-0.5);
\draw
	(0,0.5-0.1) to [out=0,in=180]
	(1,0.5) to [out=0,in=90]
	(2,0) to [out=270,in=0]
	(1,-0.5) to [out=180,in=0]
	(0,-0.5+0.1) to [out=180,in=0]
	(-1,-0.5) to [out=180,in=270]
	(-2,0) to [out=90,in=180]
	(-1,0.5) to [out=0,in=180]
	(0,0.5-0.1);
\draw (-1.5,0) to [out=-30,in=180+30] (-0.5,0);
\draw (-1.3,-0.1) to [out=30,in=180-30] (-0.7,-0.1);
\draw (0.5,0) to [out=-30,in=180+30] (1.5,0);
\draw (0.7,-0.1) to [out=30,in=180-30] (1.3,-0.1);
\draw [red] (-1,-0.15) to [out=0,in=0] node [right] {$a$} (-1,-0.5);
\node [dot] at (1.75,0) {};
\node at (0,-0.75) {$S$};
\end{scope}

\begin{scope}[shift={(2.5,0)}]
\draw
	(-1.5,0) to [out=-30,in=90]
	(-1.1,-0.35) to [out=270,in=0]
	(-1.25,-0.5) to [out=180,in=270]
	(-2,0) to [out=90,in=180]
	(-1,0.5) to [out=0,in=180]
	(0,0.5-0.1) to [out=0,in=180]
	(1,0.5) to [out=0,in=90]
	(2,0) to [out=270,in=0]
	(1,-0.5) to [out=180,in=0]
	(0,-0.5+0.1) to [out=180,in=0]
	(-0.75,-0.5) to [out=180,in=270]
	(-0.9,-0.35) to [out=90,in=210]
	(-0.5,0);
\draw (-1.3,-0.1) to [out=30,in=180-30] (-0.7,-0.1);
\draw (0.5,0) to [out=-30,in=180+30] (1.5,0);
\draw (0.7,-0.1) to [out=30,in=180-30] (1.3,-0.1);
\node [dot] at (1.75,0) {};
\node [dot] at (-1.25,-0.3) {};
\node [dot] at (-0.75,-0.3) {};
\node at (0,-0.75) {$S'$};
\end{scope}

\draw [->] (-0.35,0) -- node [above] {Crush} (0.35,0);

\end{tikzpicture}
\caption{Crushing $S$ along $a$.}
\label{fig:crush}
\end{figure}
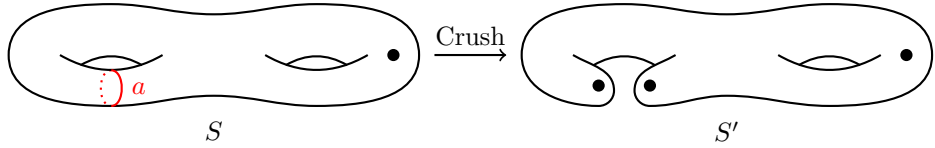

Any multicurve $x$ on $S$ has a projection $x'$ on $S'$.
If each component of $x$ meets $a$ then $x'$ is a multiarc (where parallel components are combined).
Otherwise, if $x$ is disjoint from $a$ then $x'$ is a curve.
The third possibility in which some but not all of the components of $x$ are disjoint from $a$ will not occur.

Let $\calT'_b \in \calF(S')$ be a triangulation of $S'$ which contains the multiarc $b'$.
Let
\[ T' \defeq \{ \calT' \in \calF(S') : d_{\calF(S')}(\calT', \calT'_b) \leq 2 \zeta \ell \}. \]
For each multiarc $t \subseteq \calT' \in T'$, let $m' \defeq \partial N(t)$.
Such a multicurve $m'$ has a unique lift back to a multicurve $m$ on $S$.
For example, see Figure~\ref{fig:lift}.
Define $A_1$ to be the set of all such $m$.

\begin{figure}[ht]
\centering
\begin{tikzpicture}[scale=1.35,thick]

\begin{scope}[shift={(-2.5,0)}]
\draw [blue, dotted] (-1.4,-0.05) to [out=-60,in=60] (-1.4,-0.47);
\draw [blue, dotted] (-0.6,-0.05) to [out=240,in=120] (-0.6,-0.47);
\draw
	(0,0.5-0.1) to [out=0,in=180]
	(1,0.5) to [out=0,in=90]
	(2,0) to [out=270,in=0]
	(1,-0.5) to [out=180,in=0]
	(0,-0.5+0.1) to [out=180,in=0]
	(-1,-0.5) to [out=180,in=270]
	(-2,0) to [out=90,in=180]
	(-1,0.5) to [out=0,in=180]
	(0,0.5-0.1);
\draw (-1.5,0) to [out=-30,in=180+30] (-0.5,0);
\draw (-1.3,-0.1) to [out=30,in=180-30] (-0.7,-0.1);
\draw (0.5,0) to [out=-30,in=180+30] (1.5,0);
\draw (0.7,-0.1) to [out=30,in=180-30] (1.3,-0.1);

\draw [blue] (-1.4,-0.05) to [out=210,in=270]
	(-1.65,0) to [out=90,in=180]
	(-1,0.2) to [out=0,in=90]
	(-0.35,0) to [out=270,in=-40]
	(-0.6,-0.05);
\draw [blue] (-1.4,-0.47) to [out=150,in=270]
	(-1.85,0) to [out=90,in=180]
	(-1,0.4) to [out=0,in=90]
	(-0.15,0) node [right] {$m$} to [out=270,in=40]
	(-0.6,-0.47);
\node [dot] at (1.75,0) {};
\node at (0,-0.75) {$S$};
\end{scope}

\begin{scope}[shift={(2.5,0)}]
\draw [blue, dotted] (-1.4,-0.05) to [out=-60,in=60] (-1.4,-0.49);
\draw [blue, dotted] (-0.6,-0.05) to [out=240,in=120] (-0.6,-0.49);
\draw
	(-1.5,0) to [out=-30,in=90]
	(-1.1,-0.35) to [out=270,in=0]
	(-1.25,-0.5) to [out=180,in=270]
	(-2,0) to [out=90,in=180]
	(-1,0.5) to [out=0,in=180]
	(0,0.5-0.1) to [out=0,in=180]
	(1,0.5) to [out=0,in=90]
	(2,0) to [out=270,in=0]
	(1,-0.5) to [out=180,in=0]
	(0,-0.5+0.1) to [out=180,in=0]
	(-0.75,-0.5) to [out=180,in=270]
	(-0.9,-0.35) to [out=90,in=210]
	(-0.5,0);
	
\draw [red] (-1.25,-0.3) to [out=180,in=270]
	(-1.75,0) to [out=90,in=180]
	(-1,0.3) to [out=0,in=90]
	(-0.25,0) to [out=270,in=0]
	(-0.75,-0.3);
\node [red] at (-1,-0.3) {$t$};
\draw [blue] (-1.4,-0.05) to [out=210,in=270]
	(-1.65,0) to [out=90,in=180]
	(-1,0.2) to [out=0,in=90]
	(-0.35,0) to [out=270,in=-40]
	(-0.6,-0.05);
\draw [blue] (-1.4,-0.49) to [out=150,in=270]
	(-1.85,0) to [out=90,in=180]
	(-1,0.4) to [out=0,in=90]
	(-0.15,0) node [right] {$m'$} to [out=270,in=40]
	(-0.6,-0.49);
\draw (-1.3,-0.1) to [out=30,in=180-30] (-0.7,-0.1);
\draw (0.5,0) to [out=-30,in=180+30] (1.5,0);
\draw (0.7,-0.1) to [out=30,in=180-30] (1.3,-0.1);
\node [dot] at (1.75,0) {};
\node [dot] at (-1.25,-0.3) {};
\node [dot] at (-0.75,-0.3) {};
\node at (0,-0.75) {$S'$};
\end{scope}

\draw [<-] (-0.35,0) -- node [above] {Lift} (0.35,0);

\end{tikzpicture}
\caption{Lifting $m' \defeq \partial N(t)$ on $S'$ back to $m$ on $S$.}
\label{fig:lift}
\end{figure}
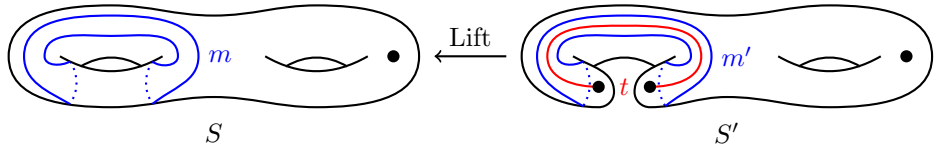

\begin{remark}
The multiarcs $t \subseteq \calT' \in T'$ are independent of the ordering of the edges of $\calT'$.
Hence we need only consider the triangulations of $S'$ that can be obtained from $\calT'_b$ by applying at most $2 \zeta \ell$ flips.
Thus we have that $|A_1| \leq 2^\zeta \cdot \zeta^{2 \zeta \ell}$.
\end{remark}

\begin{proposition}
If $a_0, \ldots, a_\ell$ is a tight path from $a$ to $b$ then $a_1 \in A_1$.
\end{proposition}

\begin{proof}
Note that every component of $a_2, \ldots, a_\ell$ meets $a$ and so these project to multiarcs $a'_2, \ldots, a'_\ell$ on $S'$.

We claim that each $a'_i$ is contained in a triangulation $\calT'_i$ such that
\[ d_{\calF(S')}(\calT'_i, \calT'_b) \leq 2 \zeta (\ell - i). \]
This is clearly true for $a'_\ell = b$ as this is contained in $\calT'_b$.

Now suppose that $\calT'_{i+1}$ is a triangulation which contains $a'_{i+1}$.
Since disjointness is preserved by projecting, $a'_i$ and $a'_{i+1}$ must both be submultiarcs of a common triangulation $\calT'_{i}$.
However, as $a'_{i+1}$ fills it decomposes the surface into polygons and once-punctured polygons.
Hence
\[ d_{\calF(S')}(\calT'_i, \calT'_{i+1}) \leq 2 \zeta \]
by the Sleator--Tarjan--Thurston diameter bound on the flip graph of a polygon \cite[Lemma~2]{SleatorTarjanThurston}.
The claim then holds by induction on $i$.

On the other hand, $a_1 = \partial N(a_0 \cup a_2)$ is disjoint from $a$ and so it projects to a curve $a'_1 = \partial N(a'_2)$ in $S'$.
By the previous claim, there is a triangulation $\calT'_2 \in T'$ which contains $a'_2$.
Therefore, $a_1 \in A_1$ by construction.
\end{proof}

\begin{proposition}
\label{prop:a_1_bounds}
Suppose that $\ell$ is fixed.
We can compute $A_1$ in polynomial time.
Moreover, for each curve $m \in A_1$ we have that
\[ ||m|| \leq \poly(||a|| + ||b||). \]
\end{proposition}

\begin{proof}
Suppose that $c \subseteq a$ is a curve.
If $c \notin \short(\calT_0)$ then we can first switch to a triangulation $\calT_1$ such that $c \in \short(\calT_1)$ and $||\calT_1(a)||, ||\calT_1(b)|| \leq \poly(||a|| + ||b||)$ in $\poly(||a||)$ time \cite[Corollary~3.8]{BellSimplifying}.
Hence, without loss of generality we may assume that $c \in \short(\calT_0)$.

Under this assumption, it is straightforward to crush $\calT_0$ along $c$ and obtain a triangulation $\calT'$ of the crushed surface.
Furthermore, since $c \in \short(\calT_0)$, we can compute minimal position representatives of $c$ and $b$ on $\calT_0$ in $\poly(||b||)$ time \cite[Corollary~2.6]{BellWebbSimplificationApplications}.
We can directly read $\calT'(a)$ and $\calT'(b')$ off of these representatives and it follows that $||\calT'(b')|| \leq ||b||$.

By repeating this crushing procedure for each component of $a$ in turn, the triangulation $\calT_0$ of $S$ descends to a triangulation $\calT'_0$ of $S'$ \cite[Page~24]{BellThesis}.
Furthermore, by tracking $b$ through this crushing process we obtain $\calT'_0(b')$ in $\poly(||a|| + ||b||)$ time.

Now the coordinates of $b'$ on $\calT'_0$ may be very complicated.
However, we can use the simplification process again to obtain a triangulation $\calT'_b$ of $S'$ which contains $b'$ in $\poly(||b||)$ time as $||\calT'_0(b')|| \leq ||b||$.

We obtain each $\calT' \in T'$ by applying at most $2 \zeta \ell$ flips and reorderings to $\calT'_b$.
By computing each $\calT' \in T'$ and each submultiarc $t \subseteq \calT'$ in turn, we can compute
\[ \{ \calT'_b(\partial N(t)) : t \subseteq \calT' \in T' \}. \]
Furthermore, as each triangulation in $\calF(S')$ has at most $\zeta \cdot \zeta!$ neighbours and $\zeta$ and $\ell$ are fixed, this can all be done in constant time.
Moreover, this shows that $||\calT'_b(\partial N(t))|| \leq O(1)$.

Reversing the simplification process of $b'$ we are able to construct
\[ \{ \calT'_0(\partial N(t)) : t \subseteq \calT' \in T' \} \]
in $\poly(||b||)$ time.
Additionally, we have that $||\calT'_0(\partial N(t))|| \leq \poly(||b||)$.

Now, the effect of lifting a curve on $S'$ back to $S$ is linear \cite[Lemma~3.4.6]{BellThesis}.
That is, there is an integral matrix $M$ such that $\calT_0(m) = M \cdot \calT'_0(m')$.
The entries of $M$ are at most $\intersection(a, \calT_0)$.
Hence we can compute
\[ A_1 = \{ M \cdot \calT'_0(\partial N(t)) : t \subseteq \calT' \in T' \} \]
in $\poly(||a|| + ||b||)$ time and for each $m \in A_1$ the bound
\[ ||\calT_0(m)|| \leq \poly(||a|| + ||b||) \]
holds as required.
\end{proof}

For each $a_1 \in A_1$ we can then compute every tight path $(a_1, \ldots, a_\ell) \in P_{\ell-1}(a_1, b)$ by induction.
Then $P_\ell(a, b)$ consists of the multipaths $(a, a_1, \ldots, a_\ell)$, where $a_1 \in A_1$ and $(a_1, \ldots, a_\ell) \in P_{\ell-1}(a_1, b)$, which are tight.
Note that it is sufficient to check that $a_1 = \partial N(a \cup a_2)$ and $a$ and $a_i$ fill for each $i \geq 3$ to determine whether such a multipath is tight.
Furthermore, as $a \notin A_1$ and $(a_1, \ldots, a_\ell)$ has length $\ell-1$, all of the tight paths that we have created have length exactly equal to $\ell$ automatically.
Hence we can compute $P_\ell(a, b)$.

\begin{remark}
\label{rem:bounded_length_bounds}
Suppose that $\ell$ is fixed.
By applying Proposition~\ref{prop:a_1_bounds} repeatedly we see that if $(a_0, \ldots, a_\ell) \in P_\ell(a, b)$ then $||a_i|| \leq \poly(||a|| + ||b||)$.
However, due the recursive nature of this construction, a priori the degree of this polynomial depends on $\ell$.
\end{remark}

\begin{corollary}
Suppose that $\ell$ is fixed.
Algorithm~\ref{alg:tight_filling_paths} computes $P_\ell(a, b)$ in polynomial time. \qed
\end{corollary}

\begin{algorithm}[ht]
\caption{Tight paths of given length.}
\label{alg:tight_filling_paths}
\begin{algorithmic}
\Require{Multicurves $a$ and $b$, and $\ell > 0$.}
\Ensure{$P_\ell(a, b)$}
	\If {$\ell = 1$} \Comment{Base case~(\ref{itm:base_1}).}
		\IIf {$\intersection(a, b) = 0$ and $a \cap b = \emptyset$}
			\Return $\{(a, b)\}$
		\EElse
			\Return $\emptyset$
		\EndIIf
	\ElsIf {$\ell = 2$} \Comment{Base case~(\ref{itm:base_2}).}
		\State $m \gets \partial N(a \cup b)$
		\IIf {$m \neq \emptyset$ and $m \cap (a \cup b) = \emptyset$}
			\Return $\{(a, m, b)\}$
		\EElse
			\Return $\emptyset$
		\EndIIf
	\Else \Comment{$\ell \geq 3$.}
	\IIf {$a$ and $b$ do not fill $S$}
		\Return $\emptyset$
	\EndIIf
	\State $A_1 \gets \emptyset$
	\State $b' \gets$ the multiarc induced on $S'$ by $b$ \Comment{$S'$ is $S$ crushed along $a$.}
	\State $\calT'_b \gets$ a triangulation containing $b'$
	\ForEach {triangulation $\calT' \in B_{\calF(S')}(\calT'_b, 2 \zeta \ell)$} \Comment{The ball about $\calT'_b$.}
		\ForEach {multiarc $t \subseteq \calT'$}
			\State $m' \gets \partial N(t)$
			\State $m \gets$ the lift of $m'$ back to $S$
			\State $A_1 \gets A_1 \cup \{m\}$
		\EndFor
	\EndFor
	\State $P \gets \emptyset$ \Comment{This will be $P_\ell(a, b)$ when we finish.}
	\ForEach {$a_1 \in A_1$}
		\ForEach {multipath $(a_1, a_2, \ldots, a_\ell) \in P_{\ell-1}(a_1, b)$} \Comment {Recurse.}
			\If {$(a, a_1, \ldots, a_\ell)$ is a tight path}
				\State $P \gets P \cup \{(a, a_1, \ldots, a_\ell)\}$
			\EndIf
		\EndFor
	\EndFor
	\State \Return $P$
	\EndIf
\end{algorithmic}
\end{algorithm}

\begin{remark}
If $\ell = d_{\calC(S)}(a, b)$ then $P_\ell(a, b)$ contains a tight geodesic from $a$ to $b$.
Unfortunately Algorithm~\ref{alg:tight_filling_paths} is at least exponential in $\ell$ and there are examples in which $d_{\calC(S)}(a, b) \approx ||a|| + ||b||$.
Therefore it alone cannot be used to compute a tight geodesic between any pair of curves in polynomial time.
\end{remark}

\subsection{Tight geodesics}

We now describe how to compute a tight geodesic between an arbitrary pair of curves $a$ and $b$ in polynomial time.
(See Algorithm~\ref{alg:geodesic}.)

We start by using Algorithm~\ref{alg:quasiconvex} to build a $K'$--quasiconvex subset $U$ of $\calC(S)$ containing $a$ and $b$ to guide our construction.
In order for this process to complete in polynomial time it is important that $U$ only contains polynomially many curves.

Now any tight geodesic between $a$ and $b$ is contained within a uniform neighbourhood of $U$.
Unfortunately, there are infinitely many curves in such a neighbourhood of $U$ and so we use a further refinement.

\begin{theorem}[{\cite[Theorem~6.2.3]{WebbThesis} \cite[Theorem~5.7]{WebbCombinatorics}}]
\label{thrm:tight_neighbourhood}
There is a computable constant $L$ such that any tight geodesic from $a$ to $b$ only involves multicurves in
\[ V \defeq \bigcup_{i=1}^L \bigcup_{u,u' \in U} P_i(u, u') \]
where this union is of all the multicurves of all of the multipaths in $P_i(u, u')$. \qed
\end{theorem}

Again, $L$ depends only on $K'$ and so is also a universal constant.

We compute $V$ by repeatedly using Algorithm~\ref{alg:tight_filling_paths} for each pair of curves in $U$.
By Lemma~\ref{lem:quasiconvex_bounds} and Remark~\ref{rem:bounded_length_bounds}, we have that $||v|| \leq \poly(||a|| + ||b||)$ for each $v \in V$.
Therefore we can also compute
\[ E \defeq \{(v, v') \in V \times V : \intersection(v, v') = 0 \; \textrm{and} \; v \cap v' = \emptyset\}, \]
which consists of all pairs of curves in $V$ that can appear consecutively in a multipath, in polynomial time \cite[Corollary~2.6]{BellWebbSimplificationApplications}.

Now let $G$ be the graph with vertex set $V$ and edge set $E$.
By construction $G$ contains all tight geodesics from $a$ to $b$.
Hence, if we use a pathfinding algorithm on $G$ to construct a geodesic $a_0, \ldots, a_n$ from $a$ to $b$ in $G$ then this sequence of multicurves will also be a geodesic multipath from $a$ to $b$.

However, $G$ may also contain non-tight geodesics and so the geodesic multipath we obtain is not necessarily tight.
Thus we finish by tightening the multicurves in the order $a_1, \ldots, a_{n-1}$ to obtain a tight geodesic from $a$ to $b$ \cite[Lemma~4.5]{MasurMinskyII}.

From the complexity bounds of Remark~\ref{rem:geodesic_bounds} it follows that this tightening can be done in polynomial time and so:

\begin{corollary}
Algorithm~\ref{alg:geodesic_path} computes a tight geodesic from $a$ to $b$ in polynomial time. \qed
\end{corollary}

\begin{algorithm}[ht]
\caption{A tight geodesic.}
\label{alg:geodesic_path}
\begin{algorithmic}
\Require{Curves $a$ and $b$.}
\Ensure{A tight geodesic from $a$ to $b$.}
	\State $U \gets$ a $K$--quasiconvex subset of $\calC(S)$ containing $a$ and $b$ \Comment{Use Algorithm~\ref{alg:quasiconvex}.}
	\State $V \gets \emptyset$
	\For {$i \gets 1,\ldots,L$}
		\ForEach {$u, u' \in U \times U$}
			\ForEach {$(a_0, \ldots, a_i) \in P_i(u, u')$} \Comment{Use Algorithm~\ref{alg:tight_filling_paths}.}
				\State $V \gets V \cup \{a_0, \ldots, a_i\}$
			\EndFor
		\EndFor
	\EndFor
	\State $E \gets \{ (v, v') \in V \times V: \intersection(v, v') = 0 \; \textrm{and} \; v \cap v' = \emptyset \}$
	\State $G \gets$ Graph $(V, E)$
	\State $a_0, \ldots, a_n \gets$ a geodesic in $G$ from $a$ to $b$  \Comment{Might not be tight.}
	\For {$i \gets 1, \ldots, n-1$}
		\State $a_i \gets \partial N(a_{i-1} \cup a_{i+1})$  \Comment{Afterwards $a_{i-1}$ is still tight.}
	\EndFor
	\State \Return $a_0, \ldots, a_n$
\end{algorithmic}
\end{algorithm}

\begin{remark}
Although $G$ only has polynomially many vertices, it can of course contain exponentially many geodesics from $a$ to $b$.
It is not possible to provide all tight geodesics from $a$ to $b$ in polynomial time.
\end{remark}

\subsection{Geodesics}

We can finally give a polynomial-time algorithm to compute a geodesic in $\calC(S)$ between a pair of curves $a$ and $b$.
(See Algorithm~\ref{alg:geodesic}.)

We start by constructing a geodesic multipath from $a$ to $b$.
We then obtain a geodesic from $a$ to $b$ in $\calC(S)$ by simply choosing an induced path of this multipath.
As we can extract the components of a multicurve in polynomial time, this can also be done in polynomial time.

\begin{remark}
\label{rem:geodesic_bounds}
By Lemma~\ref{lem:quasiconvex_bounds} and Remark~\ref{rem:bounded_length_bounds}, if $a_0, \ldots, a_n$ is a tight geodesic from $a$ to $b$ obtained by this construction then $||a_i|| \leq \poly(||a|| + ||b||)$.
However $L$ being universal only implies that the degree of this polynomial is universal.
The coefficients still depend (only) on $S$.

This shows that the general problem, in which the surface is allowed to vary and is given as part of the input, is \emph{fixed parameter tractable} \cite[Section~2.1]{DowneyFellows}.
\end{remark}

\begin{corollary}
Algorithm~\ref{alg:geodesic} computes a geodesic in $\calC(S)$ from $a$ to $b$ in polynomial time. \qed
\end{corollary}

\begin{algorithm}[ht]
\caption{A geodesic in $\calC(S)$.}
\label{alg:geodesic}
\begin{algorithmic}
\Require{Curves $a$ and $b$.}
\Ensure{A geodesic in $\calC(S)$ from $a$ to $b$.}
	\State $a_0, \ldots, a_n \gets$ a tight geodesic from $a$ to $b$ \Comment{Use Algorithm~\ref{alg:geodesic_path}.}
	\For {$i \gets 0,\ldots,n$}
		\State $c_i \gets$ a component of $a_i$
	\EndFor
	\State \Return $c_0, \ldots, c_n$
\end{algorithmic}
\end{algorithm}

\begin{remark}
Once more, although a tight geodesic from $a$ to $b$ only has polynomially many terms, it can of course contain exponentially many induced paths that are geodesic.
Again, it is not possible to provide all geodesics from $a$ to $b$ in $\calC(S)$ in polynomial time.
\end{remark}

\begin{remark}
By induction, our methods should also be able to provide a Masur--Minsky hierarchy of tight geodesics \cite[Theorem~4.6]{MasurMinskyII} in polynomial time.
\end{remark}

\section{Asymptotic translation lengths}
\label{sec:asymp_trans_len}

The mapping class group $\Mod(S)$ has a natural action on $\calC(S)$.
Given a mapping class $f \in \Mod(S)$, its \emph{asymptotic translation length} is
\[ \translationlength(f) \defeq \liminf_{i \to \infty} \frac{\dC(c, f^i(c))}{i} \]
where $c \in \calC(S)$.
This definition is independent of the choice of $c$.

In this section we use our polynomial-time algorithm for computing geodesics to give a polynomial-time algorithm to compute $\translationlength(f)$.
(See Algorithm~\ref{alg:asymptotic_translation_length}.)

We will assume throughout that $f$ is aperiodic as otherwise $\translationlength(f) = 0$ automatically.

This limit is always in fact an integral multiple of $1/D$ where $D = D(S)$ is a constant depending only on $S$ \cite[Corollary~1.5]{BowditchTight}.
This constant is explicitly bounded above by a computable constant which is at most a doubly exponential function of $\xi$ \cite[Theorem~6.3]{WebbCombinatorics}.
We note two important consequences of this choice of $D$.
First, if $\translationlength(f) \neq 0$ then $\translationlength(f) \geq 1/D$.
Second, if $\translationlength(f) \neq \translationlength(g)$ then $|\translationlength(f) - \translationlength(g)| \geq 1 / D$.

To compute $\translationlength(f)$ when $f$ is aperiodic, we set $B$ to be the constant of the bounded geodesic image theorem \cite[Theorem~3.1]{MasurMinskyII}.
The second author showed that $B$ is universal and can be taken to be $100$ \cite[Corollary~1.3]{WebbUniformBoundedGeodesic} \cite[Theorem~4.1.8]{WebbThesis}.
Now let
\[ M = M(S) \defeq 28 (\xi!) \delta B D. \]

Let $c \in \calC(S)$ be a curve and let $\gamma$ be a geodesic from $c$ to $f^M(c)$.
Finally, let $c'$ be a midpoint of $\gamma$.
We claim that if we round $\dC(c', f^M(c')) / M$ to the nearest $1 / D$ then we obtain $\translationlength(f)$.
To show this we consider two different possibilities:

\subsection{Non-zero translation}

First, suppose that $\translationlength(f) \neq 0$.
In this case $f$ cannot be reducible and so it must be pseudo-Anosov.
Let $\alpha$ be an \emph{axis} for $f^M$, that is, a biinfinite geodesic such that $f^M(\alpha) = \alpha$ \cite[Theorem~1.4]{BowditchTight}.

\begin{proposition}
\label{prop:midpoint_axis_distance}
The curve $c'$ is at most $2 \delta$ away from $\alpha$.
\end{proposition}

\begin{proof}
For ease of notation let $c_0 \defeq c$ and $c_M \defeq f^M(c)$.
Additionally let $x_0$ and $x_M$ be nearest point projections of $c_0$ and $c_M$ to $\alpha$ respectively.
Without loss of generality we may take $x_M = f^M(x_0)$.
See Figure~\ref{fig:trans_len_irreducible}.

An exercise in hyperbolic geometry shows that, since $x_0$ and $x_M$ are nearest point projections to a geodesic,  
\[ \dC(c_0, c_M) \geq \dC(c_0, x_0) + \dC(x_0, x_M) + \dC(x_M, c_M) - 24 \delta. \]

See for example \cite[Proposition~3.4]{Maher}.
Since $\dC(c_M, x_M) = \dC(c_0, x_0)$ and $\dC(x_0, x_M) \geq M / D$ we obtain that
\[ \dC(c_0, c_M) \geq 2 \dC(c_0, x_0) + M / D - 24 \delta. \]

Now consider the quadrilateral formed by geodesics $[c_0, x_0]$, $[x_0, x_M]$, $[x_M, c_M]$ and $[c_M, c_0]$.
We have that $c' \in [c_0, c_M]$ and so $c'$ must be in a $2\delta$--neighbourhood of one of the other three sides of this quadrilateral.

If $c' \in N_{2 \delta}([c_0, x_0])$ then
\[ \dC(c_0, c_M) \leq 2 \dC(c_0, c') + 1 \leq 2 \dC(c_0, x_0) + 4 \delta + 1. \]
From this we deduce that $M / D \leq 28 \delta + 1$, which contradicts our choice of $M$.  

The same argument shows that $c' \notin N_{2 \delta}([c_M, x_M])$ and so $c' \in N_{2 \delta}([x_0, x_M])$ as required.
\end{proof}

\begin{figure}[ht]
\centering
\begin{tikzpicture}[scale=2,thick]

\coordinate (c0) at (-1, 1);
\coordinate (v) at (-0.65, 0.05);
\coordinate (cp) at (0, -0.05);
\coordinate (v2) at (0.65, 0.05);
\coordinate (cM) at (1, 1);

\draw [->, red] (-2,-0.25) -- (2,-0.25) node [right] {$\alpha$};

\draw [->, blue] (c0) to [out=-85,in=160]
	(v) to [out=-20,in=180]
	(cp) to [out=0,in=200] (v2) node [above] {$\gamma$};
\draw [blue]
	(v2) to [out=20,in=265]
	(cM);

\node [dot, label=above:{$c_0 = c$}] at (c0) {};
\node [dot, label=above:{$c'$}] at (cp) {};
\node [dot, label=above right:{$c_M = f^M(c)$}] at (cM) {};

\node [dot,label=below:{$x_0$}] (x0) at (-1,-0.25) {};
\node [dot,label=below right:{$x_M = f^M(x_0)$}] (xM) at (1,-0.25) {};

\draw [decorate,decoration={brace,amplitude=3pt,mirror,raise=4pt}] (0,-0.25) -- node [right,xshift=6pt,yshift=-0pt] {$\leq 2 \delta$} (cp);

\draw [decorate,decoration={brace,amplitude=5pt,mirror,raise=4pt}] (x0) -- node [below,xshift=0pt,yshift=-12pt] {$\geq M / D$} (xM);
\draw [decorate,decoration={brace,amplitude=5pt,mirror,raise=4pt}] (c0) -- node [left,xshift=-8pt,yshift=0pt] {$d$} (x0);
\draw [decorate,decoration={brace,amplitude=5pt,raise=4pt}] (cM) -- node [right,xshift=8pt,yshift=0pt] {$d$} (xM);

\end{tikzpicture}
\caption{The midpoint of $[c, f^M(c)]$ is close to $\alpha$.}
\label{fig:trans_len_irreducible}
\end{figure}
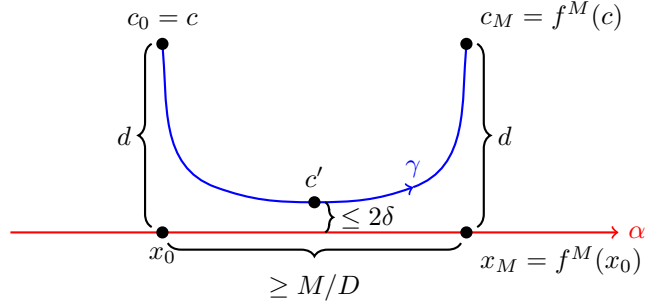

\begin{proposition}
$\frac{1}{M} (\dC(c', f^M(c')) - 4 \delta) \leq \translationlength(f)$.
\end{proposition}

\begin{proof}
Again, for ease of notation let $c'_0 \defeq c'$ and $c'_M \defeq f^M(c')$.
Additionally let $y_0$ and $y_M$ be nearest point projections of $c'_0$ and $c'_M$ to $\alpha$ respectively.
Once more, without loss of generality we may take $y_M = f^M(y_0)$.

We note that, as $\alpha$ is a geodesic which is invariant under $f^M$, we have that
\[ M \translationlength(f) = \translationlength(f^M) = \dC(y_0, y_M). \]
However, by Proposition~\ref{prop:midpoint_axis_distance} we have that $\dC(c'_0, y_0), \dC(c'_M, y_M) \leq 2 \delta$ and so
\[ \dC(y_0, y_M) \geq \dC(c'_0, c'_M) - 4 \delta. \]
Hence
\[ \translationlength(f) \geq \frac{\dC(c'_0, c'_M) - 4 \delta}{M} \]
as required.
\end{proof}

On the other hand, $\translationlength(f) \leq \dC(c', f^M(c')) / M$ trivially.
Combining these we obtain that
\[ \left|\frac{\dC(c', f^M(c'))}{M} - \translationlength(f)\right| \leq \frac{4 \delta}{M} < \frac{1}{2 D}. \]  
Hence if we round $\dC(c', f^M(c')) / M$ to the nearest $1 / D$ then we obtain $\translationlength(f)$.

\subsection{Zero translation}
\label{sub:asym_trans_len_reducible}

Second, suppose that $\translationlength(f) = 0$.

\begin{definition}[{\cite[Page~373]{FarbMargalit}}]
The \emph{canonical curve system} $\canonical(f)$ of a mapping class $f$ is the intersection of all maximal $f$--invariant multicurves.
\end{definition}

In this case $f$ must be reducible \cite[Proposition~4.6]{MasurMinskyI} and so $\canonical(f) \neq \emptyset$.
Our choice of $M$ ensures that $f^M$ fixes all components of $\canonical(f)$, including orientation.
Now recall that a component of $S - \canonical(f)$ is said to be a (non-annular) \emph{active subsurface} if the first return map of $f$ on it is pseudo-Anosov \cite[Lecture~9]{MosherLecture}.
Additionally $N(v)$, where $v \subseteq \canonical(f)$ is a curve, is an (annular) \emph{active subsurface} if the first return maps of $f$ on the components of $S - \canonical(f)$ adjacent to $v$ are all periodic \cite[Lecture~9]{MosherLecture}.

Let $m \subseteq \canonical(f)$ be a curve.

\begin{proposition}
\label{prop:reducible_distance}
The curve $c'$ is distance at most two from $m$.
\end{proposition}

\begin{proof}
Let $Y$ be a (possibly annular) active subsurface with $m \subseteq \partial Y$.
If every $v \in \gamma$ meets $m$ then every $v \in \gamma$ meets $Y$.
Therefore, by the bounded geodesic image theorem for $Y$, the projections of $\gamma$ to $\calC(Y)$ are universally bounded.
However this contradicts the choice of $M$.  
Hence there is a curve $v \in \gamma$ disjoint from $m$.

Now $f^M(m) = m$ thanks to our choice of $M$.
It follows that $\dC(v, c)$ and $\dC(v, f^M(c))$ are both at least
\[ \dC(c, m) - 1 \geq \frac{1}{2} \dC(c, f^M(c)) - 1. \]
On the other hand $\dC(c', c)$ and $\dC(c', f^M(c'))$ are both at least
\[ \frac{1}{2} \dC(c, f^M(c)) - \frac{1}{2}. \]
Therefore $\dC(c', v) \leq 1$ and so $\dC(c', m) \leq 2$ as required.
See Figure~\ref{fig:trans_len_reducible}.
\end{proof}

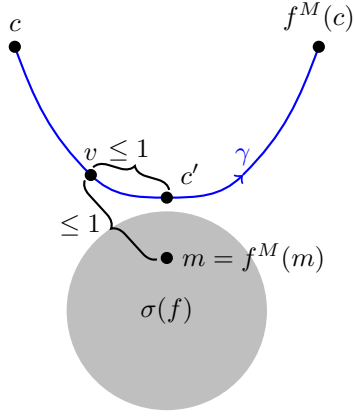
\begin{figure}[ht]
\centering
\begin{tikzpicture}[scale=2,thick]

\node [circle,minimum size=75pt, dotted, fill=lightgray] (canonical) at (0,-0.75) {$\canonical(f)$};

\coordinate (c) at (-1,1);
\coordinate (c2) at (1, 1) {};
\coordinate  (cp) at (0, 0) {};
\coordinate  (v) at (-0.5, 0.15) {};
\coordinate (v2) at (0.5, 0.15) {};

\draw [->, blue] (c) to [out=290,in=135]
	(v) to [out=-45,in=180]
	(cp) to [out=0,in=180+45] (v2) node [above] {$\gamma$};
\draw [blue]
	(v2) to [out=45,in=250]
	(c2);

\node [dot, label=above:{$c$}] at (c) {};
\node [dot, label=above:{$f^M(c)$}] at (c2) {};

\node [dot, label=above right:{$c'$}] at (cp) {};
\node [dot, label=above:{$v$}] at (v) {};
\node [dot, label=right:{$m = f^M(m)$}] (m) at (0, -0.4) {};

\draw [decorate,decoration={brace,amplitude=5pt,mirror,raise=3pt}]
(v) -- node [left,xshift=-6pt,yshift=-4pt] {$\leq 1$} (m);

\draw [decorate,decoration={brace,amplitude=5pt,raise=3pt}]
(v) -- node [above,yshift=6pt] {$\leq 1$} (cp);

\end{tikzpicture}
\caption{The midpoint of $[c, f^M(c)]$ is close to $\canonical(f)$.}
\label{fig:trans_len_reducible}
\end{figure}

Now, as $f$ acts by isometries, we have that
\[ \dC(f^M(c'), m) = \dC(c', m) \leq 2 \]
and so
\[ \dC(c', f^M(c')) \leq 4. \]
It follows that
\[ \left|\frac{\dC(c', f^M(c'))}{M} \right| \leq \frac{4}{M} < \frac{1}{2 D} \]  
and so again if we round $\dC(c', f^M(c')) / M$ to the nearest $1 / D$ then we obtain $0$, which is $\translationlength(f)$.

\subsection{Implementation}
\label{sub:asymp_trans_len_implementation}

In order to be able to do this calculation in polynomial time we must make a reasonable choice of $c$.
We take $c \in \short(\calT_0)$.

This choice of $c$ means that $||c|| \leq \zeta$ and so $||f^M(c)|| \leq \zeta M \cdot ||f||$.
These complexities are small enough that we can compute a geodesic $\gamma$ from $c$ to $f^M(c)$ in $\poly(||f||)$ time.

As above, let $c'$ be a midpoint of $\gamma$.
We then have that $||c'|| \leq \poly(||f||)$ by Remark~\ref{rem:geodesic_bounds} and so $||f^M(c')|| \leq \poly(||f||)$ too.
Once more, these bounds imply that we can compute $\dC(c', f^M(c'))$, and so $\translationlength(f)$, in $\poly(||f||)$ time.
Thus:

\begin{corollary}
Algorithm~\ref{alg:asymptotic_translation_length} computes $\translationlength(f)$ in polynomial time. \qed
\end{corollary}

\begin{algorithm}[ht]
\caption{Asymptotic translation length}
\label{alg:asymptotic_translation_length}
\begin{algorithmic}
\Require{A mapping class $f$.}
\Ensure{$\translationlength(f)$.}
	\IIf {$f$ has finite order}
		\Return 0
	\EndIIf
	\State $c \gets$ a curve in $\short(\calT_0)$
	\State $\gamma \gets$ a geodesic from $c$ to $f^M(c)$ \Comment{Use Algorithm~\ref{alg:geodesic}.}
	\State $c' \gets$ a midpoint of $\gamma$
	\State $d \gets \dC(c', f^M(c'))$ \Comment{Use Algorithm~\ref{alg:distance}.}
	\State \Return $d / M$ rounded to the nearest $1 / D$.
\end{algorithmic}
\end{algorithm}

\subsection{Nielsen--Thurston types}

Masur and Minsky's result shows that a mapping class $f$ is pseudo-Anosov if and only if $\translationlength(f) > 0$ \cite[Proposition~4.6]{MasurMinskyI}.
Therefore, by combining Algorithm~\ref{alg:asymptotic_translation_length} with a polynomial-time algorithm for computing the order of a mapping class \cite[Theorem~7.5]{FarbMargalit} \cite{MosherAutomatic}, we can also determine the Nielsen--Thurston type of a mapping class in polynomial time.
(See Algorithm~\ref{alg:nt_type}.)

\begin{corollary}
Algorithm~\ref{alg:nt_type} computes the Nielsen--Thurston type of $f$ in polynomial time. \qed
\end{corollary}

\begin{algorithm}[ht]
\caption{Nielsen--Thurston type}
\label{alg:nt_type}
\begin{algorithmic}
\Require{A mapping class $f$.}
\Ensure{The Nielsen--Thurston type of $f$.}
	\IIf {$\translationlength(f) > 0$}
		\Return Pseudo-Anosov
	\EndIIf \Comment{Use Algorithm~\ref{alg:asymptotic_translation_length}.}
	\IIf {$f$ has finite order}
		\Return Periodic
	\EndIIf
	\State \Return Reducible (and aperiodic)
\end{algorithmic}
\end{algorithm}

\begin{remark}
\label{rem:sign_asym_trans_len}
In \cite[Theorem~5.1]{GadreTsai}, Gadre and Tsai showed that if $\translationlength(f) \neq 0$ then $\translationlength(f) > 1 / R$ where $R = R(S) \defeq 18\chi(S)^2 - 30\chi(S) - 10n$.
This made explicit Masur and Minsky's constant $c(S)$ of \cite[Proposition~4.6]{MasurMinskyI}.

Therefore if we only wish to determine whether $\translationlength(f) > 0$ then it is sufficient take
\[ M \defeq \xi! B R. \]
With this value of $M$, by the argument of Section~\ref{sub:asym_trans_len_reducible}, we can then determine whether $f$ is pseudo-Anosov by checking whether
\[ \dC(c', f^M(c')) > 4. \]
Here, as before, $c'$ is a midpoint of $c$ and $f^M(c)$.
\end{remark}

\section{Canonical curve systems}

We finish by discussing an additional consequence of Section~\ref{sub:asym_trans_len_reducible} which allows us to construct invariant multicurves.
We use this in this section to give a polynomial-time algorithm for computing $\canonical(f)$.
(See Algorithm~\ref{alg:canonical}.)

We recall that $\canonical(f)$ is non-trivial if and only if $f$ is reducible and aperiodic \cite[Theorem~4.44]{Kida}.
Hence we will assume throughout that $f$ is reducible and aperiodic as otherwise $\canonical(f) = \emptyset$ automatically.

Now, throughout this section, fix
\[ N = N(S) \defeq \zeta! M, \]
where $M$ is the constant of Section~\ref{sec:asymp_trans_len}.
It will also be sufficient to take $M$ to be the constant of Remark~\ref{rem:sign_asym_trans_len}.

This choice of $N$ ensures that the \emph{canonical pieces} of $f^N$, that is, the mapping classes obtained by crushing $S$ along $\canonical(f)$, are all either pseudo-Anosov or the identity.
Again there is an explicit upper bound for $N$ in terms of $S$.

\subsection{Invariant multicurves}

We start by discussing a polynomial-time procedure for producing a multicurve which is invariant under $f^N$.
(See Algorithm~\ref{alg:invariant}.)

Let $c$ be a curve and $\gamma$ a geodesic between $c$ and $f^N(c)$.
Recall from the proof of Proposition~\ref{prop:reducible_distance} that there is a curve $v \in \gamma$ which is disjoint from a component of $\canonical(f)$.

\begin{proposition}
If $v$ misses some component of $\canonical(f)$ then $v$ and $f^N(v)$ do not fill $S$ and so $m \defeq \partial N(v \cup f^N(v)) \neq \emptyset$.
\end{proposition}

\begin{proof}
This follows immediately from the fact that $N$ was chosen so that $f^N$ preserves each component of $\canonical(f)$.
\end{proof}

\begin{proposition}
If $v$ is a curve such that $m \defeq \partial N(v \cup f^N(v)) \neq \emptyset$ then $m$ misses every active subsurface of $f^N$ and so is $f^N$--invariant.
\end{proposition}

\begin{proof}
Assume that $m$ meets a (possibly annular) active subsurface $Y$.
As $f^N$ preserves $Y$, it follows that $v$ and $f^N(v)$ must also cut $Y$.
Projecting $v$, $m$ and $f^N(v)$ to $\calC(Y)$ we see that the diameter of $\pi_Y(v) \cup \pi_Y(f^N(v))$ is bounded from above by four.
See \cite[Section~2]{MasurMinskyII} for the definition of $\pi_Y$.
However, $N$ was chosen so that $f^N$ has large translation length on $\calC(Y)$, which is a contradiction.

Now, as $N$ was chosen so that all of the canonical pieces of $f^N$ are all either pseudo-Anosov or the identity, $m$ must be $f^N$--invariant.
\end{proof}

Combining these propositions we obtain that there is a curve $v \in \gamma$ such that $m \defeq \partial N(v \cup f^N(v)) \neq \emptyset$ and such a multicurve is $f^N$--invariant.

As in Section~\ref{sub:asymp_trans_len_implementation}, we again take $c \in \short(\calT_0)$.
Once more, this choice ensures that $||c||$, $||f^N(c)||$, $||v||$, $||f^N(v)||$ and $||m||$ are all bounded above by a polynomial function of $||f||$.
Thus:

\begin{corollary}
Algorithm~\ref{alg:invariant} computes an $f^N$--invariant multicurve in polynomial time. \qed
\end{corollary}

\begin{algorithm}[ht]
\caption{Invariant multicurve}
\label{alg:invariant}
\begin{algorithmic}
\Require{An aperiodic, reducible mapping class $f$.}
\Ensure{An $f^N$--invariant multicurve $m$.}
	\State $c \gets$ a curve in $\short(\calT_0)$
	\State $\gamma \gets$ a geodesic from $c$ to $f^N(c)$ \Comment{Use Algorithm~\ref{alg:geodesic}.}
	\ForEach {$v \in \gamma$}
		\State $m \gets \partial N(v \cup f^N(v))$
		\IIf {$m \neq \emptyset$}
			\Return $m$
		\EndIIf
	\EndFor
\end{algorithmic}
\end{algorithm}

\begin{remark}
In fact from the proof of Proposition~\ref{prop:reducible_distance} we may assume that $\dC(v, c)$ and $\dC(v, f^N(c))$ are both at least
\[ \frac{1}{2} \dC(c, f^N(c)) - 1. \]
Hence we only need to actually test the middle three curves of $\gamma$ to find a suitable choice of $v$.
\end{remark}

We note that there is a slight improvement that can be made to Algorithm~\ref{alg:invariant}.
Namely, we can crush $S$ along $m$ and repeat the construction for the induced action of $f^N$ on the crushed surface.
This induced action can be computed in polynomial time \cite[Corollary~3.4.4]{BellThesis} and its complexity is at most $||f^N||$ \cite[Proposition~3.4.3]{BellThesis}.
Hence we may assume that $m$ has the maximal number of components possible.

\subsection{Canonical curve system}

We can now describe how to compute $\canonical(f)$.
To achieve this we will focus on computing $\canonical(f^N)$ however this is the same multicurve as $\canonical(f)$.
(See Algorithm~\ref{alg:canonical}.)

We start by using the improved version of Algorithm~\ref{alg:invariant} to obtain a maximal $f^N$--invariant multicurve $m$.
As $m$ is a maximal invariant multicurve it contains $\canonical(f^N)$.
Therefore it only remains to determine which components to discard.

To do this we consider all submulticurves $m' \subseteq m$.
We say that such a submulticurve $m'$ is \emph{pure} if $f^N$ induces periodic and pseudo-Anosov maps on each component of $S$ after we crush along $m'$.
The canonical curve system of $f^N$ is then the pure submulticurve of $m$ with the fewest number of components.

Since $m$ has at most $\xi$ components, there are only at most $2^\xi$ different $m'$ to check.
Once again, as $||m'|| \leq ||m|| \leq \poly(||f||)$, the induced map can be computed in $\poly(||f||)$ time \cite[Corollary~3.4.4]{BellThesis} and its complexity is at most $||f^N||$ \cite[Proposition~3.4.3]{BellThesis}.

\begin{corollary}
Algorithm~\ref{alg:canonical} computes $\canonical(f)$ in polynomial time. \qed
\end{corollary}

\begin{algorithm}[ht]
\caption{Canonical curve system}
\label{alg:canonical}
\begin{algorithmic}
\Require{A mapping class $f$.}
\Ensure{$\canonical(f)$.}
	\IIf {$f$ is pseudo-Anosov or periodic}
		\Return $\emptyset$
	\EndIIf \Comment{Use Algorithm~\ref{alg:nt_type}.}
	\State $m \gets$ a maximal $f^N$--invariant multicurve \Comment{Use Algorithm~\ref{alg:invariant}.}
	\ForEach {$m' \subseteq m$ (ordered by number of components)}
		\IIf {$m'$ is pure}
			\Return $m'$
		\EndIIf \Comment{Use Algorithm~\ref{alg:nt_type}.}
	\EndFor
\end{algorithmic}
\end{algorithm}

\begin{acknowledgements}
The authors are extremely grateful to Saul Schleimer for introducing them to these problems and for his many helpful comments and suggestions.

The first author acknowledges support from U.S. National Science Foundation grants DMS 1107452, 1107263, 1107367 ``RNMS: GEometric structures And Representation varieties'' (the GEAR Network).

The second author is supported by EPSRC Fellowship Reference EP/N019644/1.
\end{acknowledgements}

\bibliographystyle{plain}
\bibliography{bibliography}

\end{document}